\documentclass[12pt]{amsart}

\usepackage[utf8]{inputenc}
\usepackage{amstext,amsthm,amsmath}

\usepackage{latexsym}
\usepackage{amsfonts}
\usepackage{mathtools}
\usepackage{graphicx}
\usepackage[normalem]{ulem}
\usepackage{fullpage}
\usepackage{marvosym}
\usepackage[usenames,dvipsnames]{color}
\usepackage{verbatim}
\usepackage[hidelinks]{hyperref}
\usepackage[english]{babel}
\usepackage{tabularx}
\usepackage{tikz}
\usepackage{enumerate}

\usepackage{bigints}

\newtheorem*{theorem*}{Theorem}
\newtheorem*{lemma*}{Lemma}
\newtheorem*{conjecture*}{Conjecture}
\newtheorem*{question*}{Question}
\newtheorem*{problem*}{Problem}
\newtheorem*{proposition*}{Question}
\newtheorem{theorem}{Theorem}[section]
\newtheorem{definition}[theorem]{Definition}
\newtheorem{lemma}[theorem]{Lemma}
\newtheorem{proposition}[theorem]{Proposition}

\theoremstyle{definition}\newtheorem{remark}[theorem]{Remark}
\theoremstyle{definition}

\numberwithin{equation}{section}

\allowdisplaybreaks

\newcommand\CC{{\mathbb C}}
\newcommand\RR{{\mathbb R}}
\newcommand\DD{{\mathbb D}}

\newcommand\DDD{{\mathcal D}}
\newcommand\BB{{\mathbb B_n}}
\newcommand\SSS{{\mathbb S_n}}

\newcommand\Ball{{\mathbb B_2}}

\newcommand\po{{\mathtt{p}}}
\newcommand\Ba{\mathrm{B}}
\newcommand\Pol{\mathrm{P}}

\title[boundary zeros]{Boundary zeros of stable polynomials in the unit ball}
\author[D. Vavitsas]{$^\star$Dimitrios Vavitsas}
\thanks{$^\star$Corresponding author}
\address{School of Mathematics (Zhuhai), Sun Yat-Sen University, Zhuhai, Guangdong, 519082, P. R.
China}
\email{vavitsas@mail.sysu.edu.cn}
\author[J. Wu]{Jujie Wu}
\address{School of Mathematics (Zhuhai), Sun Yat-Sen University, Zhuhai, Guangdong, 519082, P. R.
China}
\email{wujj86@mail.sysu.edu.cn}
\author[K. Zarvalis]{Konstantinos Zarvalis}
\address{Department of Mathematics, Aristotle University of Thessaloniki, 54124, Thessaloniki, Greece}
\email{zarkonath@math.auth.gr}


\begin{document}

\begin{abstract}
We study polynomials in several complex variables that are stable (i.e. they have no zeros in the unit ball $\BB$), but vanish on the unit sphere along submanifolds of dimension at most one. We characterize the zeros of such polynomials on the unit sphere in terms of peak sets for the space $A^\infty(\BB).$ Furthermore, we explicitly construct a polynomial in $\CC^3$ that provides a negative answer to the question of whether the equivalent conditions of this characterization universally hold for all stable polynomials.  As an application of the developed theory, we obtain a characterization of a certain class of cyclic polynomials in the Dirichlet-type space $D_{n-\frac{1}{2}}(\BB).$  Next, we examine polynomials whose local zero set, in a neighborhood of points on the unit sphere, coincides with the graph of a holomorphic function. In this particular case, we achieve a characterization of the set of zeros lying on the unit sphere whenever the zero set has local maximum dimension $n-1$. Lastly, we discuss potential generalizations and open problems.
\end{abstract}

\maketitle
\section{Introduction}

Polynomials that are zero-free in domains of $\CC^n,$ $n\in \mathbb N,$ are often called stable and play an important role in function theory and in spaces of holomorphic functions of several complex variables. In particular, the theory of cyclic polynomials in Dirichlet-type spaces on the unit disk $\DD$, the unit ball $\BB$, and the unit polydisk $\DD^n,$ indicates that the size of the boundary zeros---namely, the set of zeros lying on the Shilov boundary of the domain---controls the cyclicity. It is therefore of fundamental importance to obtain explicit descriptions of this set.

Dirichlet-type spaces form a family $(\DDD_\alpha(\BB),||\cdot||_\alpha)$ that depends on a real parameter $\alpha\in \RR$ and consists of Hilbert spaces of holomorphic functions. An $\alpha$-Dirichlet-type space is not always an algebra, hence cyclic vectors play a role similar to that of invertible elements in algebras. Classical Hilbert spaces such as the Bergman, the Hardy, the Drury--Arveson and the Dirichlet space, are special instances of this family. Therefore, studying the aforementioned family in its full generality allows us to draw conclusions for these classical spaces.  

A natural question arising from the theory developed during the past decades is whether it is possible to cover the boundary zeros of a stable in the unit ball polynomial by suitable compact interpolation sets in the ball setting, where the corresponding interpolation functions are cyclic for appropriate parameter values, thereby leading to a characterization of cyclicity for polynomials for $\alpha \in \mathbb{R};$ see \cite{VConcrete}.

To formalize the problem, we start with a convention. From now on, a polynomial will be called \textit{stable} if it is zero-free specifically in the unit ball. Let $p\in \CC[z_1,...,z_n]$ be such a polynomial and denote by $\mathcal{Z}(p)$ its zero-set and by $\SSS$ the unit sphere in $n$ complex variables. Suppose that $\mathcal{Z}(p)\cap \SSS=\cup_iM_i$ is a finite partition of the boundary zeros, where each $M_i$ is a complex tangential submanifold in $\SSS$ and equivalent to a hypercube $(-1,1)^m$, $m=0,1,...,n-1$, via a Nash diffeomorphism; see \cite{VZ} for this description. We seek conditions under which there exist finitely or countably many interpolation pairs $\{(K^i_j,f^i_j)\}_j$ such that $\overline{M_i}\subset \cup_jK^i_j$, $\dim_\RR K^i_j\leq \dim_\RR M_i$, for all $i,j$, each $K_j^i$ is a peak set for $A^\infty(\BB)$, and the corresponding peak functions $f_j^i$ are cyclic in $\DDD_\alpha(\BB)$, where  $\alpha=\frac{2n-\mathrm{dim}_\RR M_i}{2}$. Recall that $A^\infty(\BB)$ stands for the space of holomorphic in the unit ball functions with partial derivatives that extend continuously to the unit sphere for all orders. This covering idea follows directly from \cite[Theorem 3.6]{Perfekt}, \cite[Theorem 1]{VZ}, and \cite[Lemma 25]{Ziarati}. Thorough definitions of all the above notions appear in Section \ref{sec:prelim}.

While we show that this covering approach can be successfully realized in specific scenarios, we also identify problematic cases by providing an explicit polynomial for the setting $n=3$. The theory developed throughout this paper naturally leads to the introduction of the notions of Nash stratification, essential singularities, and weak essential singularities of the one dimensional set $\mathcal{Z}(p)\cap \SSS$ in the context of interpolation theory; detailed information and precise definitions follow in Section~\ref{sec: Main results}. Roughly speaking, essential and weak essential singularities resemble the boundary point of a geometric cusp component, with the difference that the latter cannot always be detected via interpolation theory because the component can lie on a proper submanifold. For $n=2,$ the two definitions coincide due to the dimension of $\dim_\RR\mathcal{Z}(p)\cap \mathbb S_2$ of a stable polynomial. Whether a problematic case can actually occur in the setting $n=2$ remains an open question, which we explicitly pose in the final section. 

We now outline the primary results that will be stated and proved throughout the course of this manuscript. First of all, we provide a characterization of $\mathcal{Z}(p)\cap \SSS$ with respect to interpolation sets of the ball whenever $\dim_\RR \mathcal{Z}(p)\cap \SSS=1:$ considering a Nash stratification $\mathcal{Z}(p)\cap \SSS=\cup_iM_i,$ then the set $\mathcal{Z}(p)\cap \SSS$  has no essential singularities if and only if each $\overline{M_i}$ is a peak set for $A^\infty(\BB)$ as shown in Theorem~\ref{Character}. In particular, if $\mathcal{Z}(p)\cap \SSS$ has no weak essential singularities, then there exist finitely many peak sets $K_j$ for $A^\infty(\BB)$ such that $\mathcal{Z}(p)\cap \SSS=\cup_jK_j$. Each $K_j$ is contained in an one dimensional complex tangential $C^\infty$ submanifold in $\SSS$; see Theorem~\ref{theoremm}.

In light of Theorem~\ref{theoremm}, in Theorem~\ref{main2} we are able to characterize cyclic polynomials without weak essential singularities in $\DDD_{n-\frac{1}{2}}(\BB)$. 

On the other hand, we explicitly construct a polynomial $\pi$ in $\CC^3$ such that $\dim_\RR\mathcal{Z}(\pi)\cap \mathbb S_3=1$ and the set $\mathcal{Z}(\pi)\cap \mathbb S_3$ contains a weak essential singularity which is not essential. This polynomial shows the boundaries of this framework with respect to the problem of cyclicity for polynomials in the setting $n\geq 3$. More specifically, a straightforward application of Theorems \ref{theoremm} and \ref{main2} seems inherently insufficient for the characterization of cyclicity for all stable polynomials $p$ with $\dim_\RR\mathcal{Z}(p)\cap \SSS=1$.

Next, we study polynomials whose local zero set, in a neighborhood of points of the unit sphere, coincides with the graph of a holomorphic function. We show that the closure of a Nash component of the boundary zero set is contained locally in real analytic, totally real submanifolds lying in $\SSS$; see Proposition~\ref{prop graph}. Furthermore, if the Nash component is of maximum real dimension $n-1$, then its closure is a peak set for $A^\infty(\BB)$ as proved in Proposition~\ref{prop}. If additionally, the boundary zeros have maximum local dimension $n-1$,  then  Proposition~\ref{prop} becomes a characterization. In particular, for the setting $n=2$, we characterize the boundary zeros, that is, the boundary zeros of a stable polynomial whose zero set is $1$-dimensional and locally the graph of a holomorphic function, admit no weak essential singularities. This result is containd in Theorem~\ref{n=2 characterization zero set graph}; see also \cite{Knese} for a different approach. Moreover, we provide a characterization of cyclicity for a specific class of these polynomials in the setting $n>2$; see Theorem~\ref{cyclicity graph}. We also refer to \cite{KV} for the characterization of any stable polynomial in the setting $n=2$.

Finally, we deal with certain cases concerning the geometric nature of the Nash components $M_i$ of $\mathcal{Z}(p)\cap \SSS =\cup_iM_i.$ We provide a construction of a covering $\{K_j\}_j$ of $M_i$ that forms an increasing sequence of semi-algebraically connected compact peak sets for $A^\infty(\BB),$ and each $K_j$ is contained in two distinct $C^\infty$ complex tangential submanifolds, each of the minimal possible dimension. This description is the content of Theorem~\ref{path conn peak sets}.

As is evident from all the above exposition, in brief, the aim of this work is to study the geometric structure of the boundary zeros with respect to interpolation sets for the unit ball, motivated by the cyclicity problem for polynomials in Dirichlet-type spaces. The proofs, all of which are presented in Section \ref{sec:proofs}, rely on techniques stemming from semi-algebraic geometry and interpolation theory on the unit ball. These include foundational work in semi-algebraic geometry by F. Bruhat, H. Cartan, J. Milnor, R. Robson  (\cite{Cartan, Milnor, Robson}) as well as developments in interpolation theory by J. Bruna J. Chaumat, A.M. Chollet, J.E. Fornaess, F.R. Harvey, B.S. Henriksen,  J.M. Ortega, and R.O. Wells (\cite{Bruna,Chollet,Chaumat,Fornaess,HW}) together with more recent results concerning the geometric nature of the boundary zeros by K. Bickel, G. Knese, J.E. Pascoe, and A. Sola (\cite{Bickel,Sola1,Knese}).  

Before proceeding to the core of the article, we lay out the various studies appearing in the literature that primarily motivated this present work. The characterization of cyclic polynomials in Dirichlet-type spaces on the bidisk was established in \cite{Beneteau} by C. B\'{e}n\'{e}teau, G. Knese, \L. Kosi\'{n}ski, C. Liaw, D. Seco and A. Sola. The anisotropic counterpart was subsequently characterized by G. Knese, \L. Kosi\'{n}ski,  T. Ransford and A. Sola in \cite{KneseK}. Boundary zeros of polynomials on the polydisk were investigated in \cite{Sola1} by K. Bickel, J.E. Pascoe and A. Sola. Cyclic polynomials in Dirichlet-type spaces on the unit ball of $\CC^2$ were characterized in \cite{KV} by Ł. Kosiński et al., while necessary and sufficient conditions for cyclicity in Dirichlet-type spaces on the unit ball of $\CC^n$ were established in \cite{Perfekt} by A. Aleman, K. Perfekt, S. Richter, C. Sundberg and J. Sunkes. An explicit description of boundary zeros of stable polynomials, together with a conjecture on the characterization of cyclic polynomials in Dirichlet-type spaces on the unit ball of $\CC^n$ was given in \cite{VZ}. More recently, a study of zeros of holomorphic functions by P. Kumar and J. Sampat occurs in \cite{Jeet}, while the description of boundary zeros by G. Knese, J.E. Pascoe, and A. Sola appears in \cite{Knese}. Along with the study of cyclicity for holomorphic functions in neighborhoods of the unit ball in \cite{Ziarati} by P.T. Ziarati, they collectively provide a fertile basis for further study of the boundary zeros. As a matter of fact, a series of questions arising from the aforementioned papers in conjuction with the work developed in the present article, is presented in Section \ref{sec:further}.

\section{Preliminaries}\label{sec:prelim}
We start by presenting the notation that will be employed throughout the course of this work. For the rest of this section, we fix some $n\in\mathbb{N}$. For points $z=(z_1,\dots,z_n)$ and $w=(w_1,\dots,w_n)$ in $\CC^n$, we denote their \textit{Euclidean inner product} by $\langle z,w\rangle=z_1\overline{w}_1+\dots+z_n\overline{w}_n$ and its associated \textit{Euclidean norm} by $||z||=\sqrt{|z_1|^2+\dots+|z_n|^2}$. 

We write $\BB\vcentcolon=\{z\in\CC^n:||z||<1\}$ for the \textit{unit ball} with topological boundary the \textit{unit sphere} $\SSS\vcentcolon=\{z\in\CC^n:||z||=1\}$. Moreover, for the ball and the polydisk with center $x\in \CC^k$ and radius $\delta>0$, we use the notations $\Pol(x,\delta)$ and $\Ba(x,\delta)$, respectively.

On another note, for the zero set of a function $f\in \CC[z_1,...,z_n]$ we use the notation $\mathcal{Z}(f)\vcentcolon=\{z\in \CC^n:f(z)=0\}$. In particular, a polynomial $p\in \CC[z_1,...,z_n]$ is said to be \emph{stable (with respect to the unit ball)} if $p$ is zero-free in the unit ball, that is, $\mathcal{Z}(p)\cap \BB=\emptyset$.  The zeros of a stable polynomial that lie on the unit sphere are called \emph{boundary zeros}, and the corresponding set is written as $\mathcal{Z}(p)\cap \SSS$.

Moving on to function spaces, we denote the space of holomorphic in the unit ball functions by $\textup{Hol}(\BB)$. The space $A^k(\BB)$ consists of all holomorphic functions on the unit ball whose partial derivatives up to order $k\in \mathbb N$ extend continuously to the closure $\overline{\BB}$. In case $k=0$, we write $A(\BB)\vcentcolon=A^0(\BB)$ for the ball algebra. Furthermore, the space $A^\infty(\BB)$ consists of all holomorphic functions on the unit ball whose partial derivatives of all orders extend continuously to the closure of the unit ball. Finally, we denote by $\mathcal{O}(\overline{\BB})$ the class of all functions holomorphic in a neighborhood of $\overline{\BB}$. 

Let $f\in \mathrm{Hol}(\BB)$. The geometry of the unit ball allows the function $f$ to be represented by a power series expansion of the form
\begin{equation}\label{eq:power series}
f(z)=\sum\limits_{|k|=0}^{+\infty}a_kz^k, \quad\quad z=(z_1,\dots,z_n)\in\BB,    
\end{equation}
where $k=(k_1,\dots,k_n)$ is a multi-index of non-negative integers, $|k|=k_1+\dots+k_n$ and $z^k=z_1^{k_1}\cdots z_n^{k_n}$. We refer to \cite{Horm} and \cite{Rudin} for a detailed background on holomorphic functions on the unit ball.

\subsection{Dirichlet-type spaces}
Fix a real parameter $\alpha\in\RR$ and suppose that $f\in\mathrm{Hol}(\BB)$ has a power series expansion as in \eqref{eq:power series}. We say that $f$ belongs to the \textit{Dirichlet-type space} $\DDD_\alpha(\BB)$ provided
\begin{equation*}
  ||f||^2_\alpha\vcentcolon=\sum\limits_{|k|=0}^{+\infty}(n+|k|)^\alpha\frac{(n-1)!k!}{(n-1+|k|)!}|a_k|^2<+\infty,  
\end{equation*}
where $k!=k_1!\cdots k_n!$. The family of Dirichlet-type spaces depends on a real parameter and consists of Hilbert spaces of holomorphic functions. The $\alpha$-Dirichlet-type norm is determined by the Taylor coefficients of $f$ together with weights depending on the real parameter $\alpha$. Special cases of this family include classical Hilbert spaces of holomorphic functions on the unit ball of $\CC^n$. More specifically, the spaces  $\DDD_{-1}(\BB), \DDD_0(\BB), \DDD_{n-1}(\BB)$ and $\DDD_n(\BB)$ coincide with the well known Bergman, Hardy, Drury-Arveson and Dirichlet spaces, respectively. For further information on these spaces, we refer the interested reader to \cite{Perfekt, Brown-Shields, Primer, Hartz, Zhu}.

Fundamental properties of Dirichlet-type spaces include the following: the density of polynomials in each space $\DDD_\alpha(\BB)$, $\alpha\in \RR$; for each $f\in \DDD_\alpha(\BB),$ the functions also $z_i\cdot f$ belong to $\DDD_\alpha(\BB),$ $i = 1, ..., n$;  the \textit{point evaluation functional} with respect to a fixed point $z\in \BB$ defined by $x^*_z\vcentcolon\DDD_\alpha(\BB)\rightarrow \CC$, with $x^*_z(f)\vcentcolon=f(z)$, is a continuous linear functional.

\subsection{Cyclic vectors}
Consider the \textit{shift operator} $T_i\vcentcolon\DDD_\alpha(\BB)\to \DDD_\alpha(\BB)$, $i=1,..,n$, defined by $T_i(f)(z)\vcentcolon= z_i\cdot f(z)$. A function $f\in \DDD_\alpha(\BB)$ is called a \textit{cyclic vector} (or \emph{cyclic function}) for $\DDD_\alpha(\BB)$ if the closed invariant subspace
\begin{equation}\label{cyclic def}
[f]_\alpha:=\textrm{clos span}\{z_1^{k_1}\cdots z_n^{k_n}f:k_1,\dots,k_n=0,1,2,\dots\}
\end{equation}
coincides with the whole space $\DDD_\alpha(\BB).$ The closure in \eqref{cyclic def} is taken with respect to the norm of $\DDD_\alpha(\BB).$ Equivalently, $f\in \DDD_\alpha(\BB)$ is cyclic if and only if there exists a sequence of polynomials $p_n\in \CC[z_1,...,z_n]$ such that $p_nf\rightarrow 1$ in the $\alpha$-Dirichlet-type norm. Trivial examples of cyclic functions are the non-zero constant functions.

Cyclic functions for $\DDD_\alpha(\BB)$ must be zero-free in the unit ball. This is a consequence of the continuity of the point evaluation functionals mentioned previously. Therefore, the study of cyclic vectors for Dirichlet-type spaces focuses on functions whose zeros lie outside the unit ball. In particular, the cyclicity of a polynomial is closely related to the size of its boundary zeros. As we shall see in the sequel, the study of this boundary zero set in the ball setting requires techniques from semi-algebraic geometry and interpolation theory. See \cite{Perfekt, Beneteau, Brown-Shields, Primer, Hartz,  Knese, KV, Mironov, Zhu} and the references within them for a background on cyclicity.

\subsection{Semi-algebraic geometry}
We proceed to a rundown of notions borrowed from semi-algebraic geometry and deemed necessary for the purposes of this article. Fix $N\in\mathbb{N}$. A set $A\subset\RR^N$ is said to be \textit{semi-algebraic} provided there exists a finite number of real polynomials $f_{i,j}$,  $i=1,\dots,s$, $j=1,\dots,r_i$, such that
\begin{equation*}
 A=\bigcup\limits_{i=1}^{s}\bigcap\limits_{j=1}^{r_i}\{x\in \RR^N:f_{i,j}(x) *_{i,j} 0\},   
\end{equation*}
where $*_{i,j}$ is either $<$ or $0$.

To continue with, we are in need of certain concepts with regard to Nash submanifolds and Nash diffeomorphisms; see \cite[Chapter 2, Subsection 9.7]{Algebraic} for a concise presentation.
\begin{enumerate}
    \item[(i)] Let $A\subset\RR^M$ and $B\subset\RR^N$ be two semi-algebraic sets. A mapping $f\vcentcolon A\to B$ is characterized as \textit{semi-algebraic} if its graph is a semi-algebraic set of $\RR^{M+N}$.
    \item[(ii)] Let $A\subset\RR^N$ be semi-algebraic. A semi-algebraic mapping $f\vcentcolon A\to\RR$ of class $C^\infty$ is called a \textit{Nash function}. Moreover, given two semi-algebraic sets $A,B\subset\RR^N$, a semi-algebraic bijection $f\vcentcolon A\to B$ of class $C^\infty$ is called a \textit{Nash diffeomorphism}.
    \item[(iii)] A semi-algebraic subset $M$ of $\RR^N$ is said to be a \textit{Nash submanifold} of $\RR^N$ of dimension $d$ if for every point $x\in M$, there exists a Nash diffeomorphism $\phi$ from an open semi-algebraic neighborhood $\Omega$ of the origin in $\RR^N$ onto an open semi-algebraic neighborhood $\Omega'$ of $x$ in $\RR^N$ such that $\phi(0)=x$ and $\phi((\RR^d\times\{0\})\cap\Omega)=M\cap\Omega'$.
    \item[(iv)] Two Nash submanifolds $M$ and $N$ are considered to be \textit{Nash diffeomorphic} if there exists a bijection $f\vcentcolon M\to N$ such that both $f$ and $f^{-1}$ are Nash functions. Note that Nash functions are real analytic functions.
    \item[(vi)] Let $Y$ be a Nash submanifold of $\RR^N$.  A \emph{Nash wing} with axis $Y$ is a Nash mapping $w\vcentcolon(-1,1)\times Y\rightarrow\RR^N$ which is a homeomorphism onto its image and enjoys the following properties for each $y\in Y$:
    \begin{enumerate}
        \item $w(0,y)=y$;
        \item $\partial_tw(t,y)\neq0$, for all $t\neq0$;
        \item there is a $q\in\mathbb{N}$ such that $\partial_t^q w(0,y)\ne0$ for all $y\in Y$ whereas $\partial_t^i w(0,y)=0$ for all $0<i<q$ and all $y\in Y$.
    \end{enumerate}
\end{enumerate}
All the definitions can be extended for any real closed field in the place of $\RR$. We choose to write everything in terms of $\RR$ because it better suits our current work. We refer the interested reader to the books \cite{Algebraic} and \cite{Krantz,Lee} for a thorough background on semi-algebraic geometry and on real-analytic, smooth manifolds, respectively.

The \textit{dimension} of a set  $E\subset\RR^{2n}$ may be defined as
\begin{equation*}
  \textup{dim}(E)\vcentcolon=\max\{\textup{dim}(\Gamma):\Gamma\subset E, \Gamma \text{ submanifold of }\RR^{2n}\}.  
\end{equation*}

Concerning sets of complex vectors, any $M\subset\CC^n$ may be considered as a set $E\subset\RR^{2n}$. Through this correspondence, we define the \textit{real dimension} of such an $M$ as 
\begin{equation*}
   \textup{dim}_\RR(M)\vcentcolon=\textup{dim}(E)\in\{-1,0,1,\dots,2n\}. 
\end{equation*}
The case $\textup{dim}(E)=-1$ is devoted to the instance when $E=\emptyset$.

\begin{remark}
Given a polynomial $p\in\CC[z_1,\dots,z_n]$, the set $\mathcal{Z}(p)\cap\SSS$ may be viewed as a semi-algebraic set in $\RR^{2n}$; indeed, it is the intersection of the unit sphere $\SSS=\{z\in\CC^n: 1-||z||^2=0\}$ and the sets $\{z\in\CC^n: \textup{Im }p(z)=0\}$, $\{z\in\CC^n: \textup{Re }p(z)=0\}$. Therefore, in any case, the boundary zeros of a polynomial may be regarded as a semi-algebraic set in $\RR^{2n}$.
\end{remark}

\subsection{Interpolation theory in the unit ball}
Before moving on to the main body of this work, let us recall certain notions with respect to interpolation theory of the unit ball; see \cite[Chapter 10]{Rudin} for information about function theory of the unit ball. Fix a compact set $K\subset\SSS$.
\begin{enumerate}
    \item[(i)] $K$ is a \textit{(Z)-set} (\textit{zero set}) for $A(\BB)$ if there exists a function $f\in A(\BB)$ such that $f(\zeta)=0$ on $K$, but $f(z)\ne0$, for all $z\in\overline{\BB}\setminus K$.
    \item[(ii)] $K$ is a \textit{(PI)-set} (\textit{peak-interpolation set}) for $A(\BB)$ if it satisfies the following property: to each continuous function $g:K\rightarrow\CC$ that is not identically zero corresponds some $f\in A(\BB)$ such that $f(\zeta)=g(\zeta)$ on $K$ and $|f(z)|<\max_{\zeta\in K}|g(\zeta)|$, for all $z\in\overline{\BB}\setminus K$.
    \item[(iii)] $K$ is a called \emph{peak set} for $A^\infty(\BB)$ if there exists $f\in A^\infty(\BB)$ such that $f(\zeta)=0$ for all $\zeta\in K$ and $\mathrm{Re} f(z)>0$ for all $z\in \overline{\BB}\setminus K$. Equivalently, $K$ is a peak set for $A^\infty(\BB)$ if and only if there exists $f\in A^\infty(\BB)$ such that $f(\zeta)=1$, for all $\zeta\in K$, and $|f(z)|<1$, for all $z\in \overline{\BB}\setminus K$. In any case, the function $f$ is called the \emph{peak function} corresponding to $K$. In particular, the pair $(K,f)$ is said to be an \emph{interpolation doublet}.
\end{enumerate}

\begin{remark}\label{equivalence-Z-PI}
\cite[Theorem 10.1.2]{Rudin} yields that $K$ is a (Z)-set for $A(\BB)$ if and only if it is a (PI)-set for  $A(\BB)$. Furthermore, if $K$ is a (Z)-set for $A(\BB)$, then every compact subset $F$ of $K$ must also be a (Z)-set for $A(\BB)$.
\end{remark}

\begin{remark}
If $p\in\CC[z_1,...,z_n]$ is a stable polynomial, then the boundary zeros form a $(Z)$-set.
\end{remark}

We end this subsection by recalling the definition of complex tangential maps.
\begin{enumerate}
    \item[(i)] Pick a point $\zeta\in \SSS.$ The \emph{tangent space} $T_\zeta(\SSS)$ with respect to $\zeta$ and $\SSS$ consists of all vectors $w\in \CC^n$ that are perpendicular to the radius of $\BB$ which ends at $\zeta,$ that is, $w\in T_\zeta(\SSS)$ if and only if $\mathrm{Re }\langle w,\zeta\rangle=0.$
    \item[(ii)]  The corresponding \emph{complex tangent space} $T_\zeta^{\CC}(\SSS)$ is defined to be $T_\zeta(\SSS)\cap(iT_\zeta^{\CC}(\SSS)),$ and consists of all $w$ that satisfy $\langle w,\zeta\rangle=0.$
    \item[(iii)] Let $I\subset\RR$ be an interval on the real line. A $C^1$ curve $\gamma:I\to\SSS$ is said to be a \textit{complex tangential curve} when
$$\langle\gamma'(t),\gamma(t)\rangle=0, \quad\quad\textup{for all }t\in I.$$
That means $\gamma'(t)\in T_{\gamma(t)}^\CC(\SSS),$ for all $t\in I.$
\item[(iv)] Let $\Omega\subset\RR^m$, $m\in\mathbb{N}$, be an open set and let $\Phi:\Omega\to\SSS$ be a $C^1$ map.  Set $M:=\Phi(\Omega)$ and assume that $\Phi$ is one-to-one. Denote the Jacobian of $\Phi$ at $x\in \Omega$ by $\Phi'(x).$ We shall say that $\Phi$ is nonsingular if the rank of $\Phi'(x)$ is $m$ for every $x\in \Omega$.  Put $\zeta=\Phi(x).$  The tangent space $T_\zeta(M)$ with respect to $\zeta$ and $M$ consists of all vectors $\Phi'(x)h,$ as $h$ ranges over $\RR^m.$ The set $M$ is called complex tangential if $T_\po(M)\subset T_\po^{\CC}(\SSS),$ for all points $\po\in M.$ Equivalently, the map $\Phi$ is complex tangential if and only if for any $C^1$ curve $\gamma:[0,1]\rightarrow \Omega$ the map $\Phi\circ\gamma$ is complex tangential (see, e.g., [22, section 10.5.2]).
\end{enumerate}

\section{Auxiliary Results}

This section is devoted to results that will be used throughout the course of the proofs.
Firstly, \cite{VZ} provides a geometric description for the boundary zeros of stable polynomials.

\begin{theorem}[{\cite[Theorem 1]{VZ}}]\label{zeroset}
Let $p\in\CC[z_1,\dots,z_n]$ be a stable polynomial. If $\mathcal{Z}(p)\cap\SSS$ is non-empty, then it is either a finite set, or $\mathcal{Z}(p)\cap\SSS=\cup M_i$; a finite disjoint union of Nash submanifolds $M_i\subset\RR^{2n}$, where each $M_i$ is Nash diffeomorphic to  $(-1,1)^m$, $m=0,1,...,n-1$. In particular, each Nash diffeomorphism $\phi_i:(-1,1)^m\rightarrow M_i$ is complex-tangential and real analytic.
\end{theorem}

According to the conjecture of \cite{VZ} and the following theorem due to  A. Aleman, K. Perfekt, S. Richter, C. Sundberg and J. Sunkes, the geometry of $\mathcal{Z}(p)\cap \SSS$ determines whether a polynomial $p$ is cyclic in $\DDD_{\alpha}(\BB)$. Recall that a function $\phi:\BB\rightarrow\CC$ is called a \emph{multiplier} for the space $\DDD_{\alpha}$ if $\phi f\in \DDD_\alpha(\BB)$, for all $f\in \DDD_\alpha(\BB)$.

\begin{theorem}[{\cite[Theorem 3.6]{Perfekt}}]\label{Perf}
Let $f$ and $g$ be multipliers in $\DDD_{\alpha}(\BB)$ and suppose that\\[4pt]
$\mathrm{(i)}$ $f$ and $g$ are zero-free in $\BB;$\\[3pt]
$\mathrm{(ii)}$ $f$ extends to be analytic in a neighbourhood of $\overline{\BB};$\\[3pt]
$\mathrm{(iii)}$ $g$ satisfies a Lipschitz condition of order $a>0;$\\[3pt]
$\mathrm{(iv)}$ $\mathcal{Z}(f)\cap\SSS\subset \mathcal{Z}(g)\cap \SSS.$\\[4pt]
If $g$ is cyclic in $\DDD_{\alpha}(\BB)$, then so is $f$.
\end{theorem}

Theorems \ref{zeroset} and \ref{Perf} naturally lead to the following question.

\begin{question*}[A]\phantomsection\label{Question A}
Given a stable polynomial $p\in \CC[z_1,...,z_n],$ is it possible to find a proper covering of interpolation doublets $(K_\ell, g_\ell)$ such that\\[4pt]
$\mathrm{(i)}$ $\mathcal{Z}(g_\ell)\cap \overline{\BB}=K_\ell\subset \SSS;$\\[3pt]
$\mathrm{(ii)}$ $\mathcal{Z}(p)\cap \SSS\subset \cup K_\ell;$\\[3pt]
$\mathrm{(iii)}$ $\dim_\RR K_\ell\leq \dim_\RR\mathcal{Z}(p)\cap \SSS;$\\[3pt]
$\mathrm{(iv)}$ $g_\ell$ is a cyclic function in some  $\DDD_{\alpha}(\BB);$\\[4pt] 
thereby deducing cyclicity for $p$ in $\DDD_{\alpha}(\BB)?$
\end{question*}

Let us first work out the condition $\mathcal{Z}(g_\ell)\cap \overline{\BB}=K_\ell\subset \SSS.$ This shows that possible candidates for interpolation doublets  $(g_\ell,K_\ell)$ are pairs consisting of compact interpolations sets $K_\ell$ in the unit ball and their corresponding interpolating functions $g_\ell$. Interpolation sets may be sets like (Z)-sets, (PI)-sets, peak sets etc. The goal is to pick the right interpolation doublets that satisfies the nicest conditions possible.

As we have already mentioned, the interpolating function $g_\ell$ obtained from $K_\ell$ must be cyclic in $\DDD_{\alpha}(\BB).$ The following lemma due to P.T. Ziarati, identifies cyclicity for smooth peak functions.

\begin{lemma}[{\cite[Lemma 25]{Ziarati}}]\label{LemmaPuri}
Let $K\subset \SSS$ be a peak set for $A^\infty(\BB)$  that is contained in an $m$-dimensional $C^\infty$ complex tangential manifold in $\SSS.$ Then, there exists a peak function $g\in A^\infty(\BB)$ with respect to $K$ which is cyclic in $\DDD_{\alpha}(\BB)$ for $\alpha\leq \frac{2n-m}{2}.$ 
\end{lemma}

An application of Hopf's Lemma (see \cite[Proposition 12.2]{Fornaess book}) yields $|g(z)|\geq C\mathrm{dist}(z,\SSS),$ $z\in \BB,$ for some uniform constant $C>0.$ However, this estimate is not sufficient to identify cyclicity. These peak functions behave locally around a point of $K$ like a peak function $g(z)=1-z_n+\epsilon(z),$ where $\epsilon(z)=o(||z-(0,0,..,1)||)$ near $z\in (0,...,1)$. Furthermore, J. Bruna and J.M. Ortega in \cite[Theorem 6.2 and Lemma 6.3]{Bruna} obtain sharp estimates of the behaviour of such peak functions in the unit ball which may be compared with the ones in \cite[Theorem 21]{Chollet} and \cite{Ziarati}. So, it is reasonable to focus on interpolation doublets $(g_\ell, K_\ell)$ for $A^\infty(\BB),$ where $K_\ell$ have to be contained globally in a complex tangential manifold of the minimal possible dimension.

On the other hand, the theorem below gives a sufficient condition of identifying non-cyclicity for polynomials and shows that the compact interpolating sets $K_\ell$ must satisfy the condition $(iii)$ of \hyperref[Question A]{Question (A)}, that is,  $\dim_\RR K_\ell\leq \dim_\RR\mathcal{Z}(p)\cap \SSS.$

\begin{theorem}[{\cite[Theorem 3]{VZ}}]\label{non-cyclicity}
Let $p\in\CC[z_1,\dots,z_n]$ be a stable  polynomial. Suppose that $\mathcal{Z}(p)\cap\SSS$ contains a real submanifold in $\RR^{2n}$ of dimension $m\in\{1,...,n-1\},$ but no submanifold of any higher dimension. Then $p$ is not cyclic in $D_\alpha(\BB)$ whenever $\alpha>\frac{2n-m}{2}$.
\end{theorem}

Next, let us recall theorems regarding peak sets. J. Chaumat and A.M. Chollet identified a condition for a compact set to be a peak set for $A^\infty(\BB).$

\begin{theorem}[{\cite[Theorem 21]{Chollet}}]\label{RudinThm}
If $N\subset \SSS$ is a $C^\infty$ manifold that is complex tangential, then each compact $K\subset N$ is a peak set for $A^\infty(\BB).$
\end{theorem}

In the above statement, a $C^\infty$ manifold $N\subset \SSS$ is called complex tangential if it is locally complex tangential; that is, for every point $\po\in M$ there exists a local chart $(\Omega,\Phi)$ with respect to $\po$ such that the set $\Phi(\Omega)\subset M$ is complex tangential. It is also well known that $C^\infty$ complex tangential manifolds in the sphere are totally real and have
dimension at most $n- 1,$ see \cite[Theorem 10.5.4-10.5.6]{Rudin}. Theorem~\ref{RudinThm} shows that each compact set $K\subset M_i,$ where $M_i$ is a component of $\mathcal{Z}(p)\cap \SSS$ is a peak set for $A^\infty(\BB).$ 

Let us recall that a compact set $K$ is locally contained in $C^\infty$ complex tangential manifolds in $\SSS$ if for each $\zeta\in K,$ there exists an open neighbourhood $U\subset\CC^n$ of $\zeta$ and a $C^\infty$ complex tangential manifold $N$ in $\SSS$ such that $K\cap U$ lies in $N.$

\begin{theorem}[{\cite[Theorem 7]{Chaumat}}]\label{peak contained locally}
If $K\subset \SSS$ is a peak set for $A^\infty(\BB),$ then it is contained in $C^\infty$ complex tangential manifolds of dimension $n-1$ in $\SSS.$
\end{theorem}

See also \cite[Theorem 10.7.3 and 10.7.6]{Rudin} for the above theorem. Recall the partition $\mathcal{Z}(p)\cap \SSS=\cup_iM_i$ obtained in Theorem~\ref{zeroset} and note that $\overline{M_i}$  are also semi-algebraic sets. Therefore, it is reasonable to pose the following questions.

\begin{question*}[B]\phantomsection\label{Question B}
Is it possible to deduce that the closure of each Nash component, $\overline{M_i},$ is a peak set for $A^\infty(\BB)$? Is it possible to cover the boundary zeros by finitely many or a proper sequence of compact peak sets for $A^\infty(\BB)?$
\end{question*}

These questions are the main motivation of the present work. As we shall see, it is possible to pick proper coverings  whenever the boundary zero set satisfies additional properties yet at the same time we explicitly construct a polynomial
in $\CC^3$ that provides a negative answer to the question of whether the above finite covering universally hold for all stable polynomials.

The problem is that it is not clear for the moment what is the behaviour of the components $M_i$ near their boundaries; see Figure 1 for possible local illustrations. Furthermore, $\overline{M_i}$ might not be a compact manifold in $\SSS,$ like the boundary zeros of the model polynomials $\pi_m(z):=1-m^{m/2}z_1\cdots z_m,$ $m=1,2,...,n;$ in this case Theorem~\ref{RudinThm} would be enough. See \cite{A.Sola, V} about the cyclicity of the model polynomials $\pi_m$. See also the characterization in \cite{Ziarati} for a more general case. A priori, a complex tangential $C^\infty$ manifold $N\subset \SSS$ may be even dense in the unit sphere.

The key point is that the components $M_i$ form a  partition of the semi-algebraic set $\mathcal{Z}(p)\cap \SSS$ and they are not arbitrary $C^\infty$ complex tangential manifolds in $\SSS.$ Hence, one may apply tools stemming from semi-algebraic geometry. We shall make use of some crucial theorems in the proof of the main results.

The first concerns the characterization of peak sets for $A^\infty(\BB),$ which were proved by J.E. Fornæss and B.S. Henriksen.

\begin{lemma}[{\cite[Lemma 4.1]{Fornaess}}]\label{FornLemma}
If $K\subset \SSS$ is compact and contained in $N\cup M,$ where $N$ and $M$ are $C^\infty$ complex tangential manifolds such that $\dim_\RR N<\dim_\RR M,$ $N\cap M$ open in $N,$ and $K\cap N,$ $K\cap M$ open in $K,$ then $K$ is a peak set for $A^\infty(\BB).$
\end{lemma}

Applying the lemma above, they proved a sufficient condition of identifying peak sets for $A^\infty(\BB).$

\begin{theorem}[{\cite[Theorem 4.2]{Fornaess}}]\label{ThmF}
If a compact set $K\subset \SSS$ is locally contained in $C^\infty$ complex tangential manifolds in $\SSS$, then $K$ is a peak set for $A^\infty(\BB).$
\end{theorem}

\begin{figure*}
\centering
\includegraphics[scale=0.6]{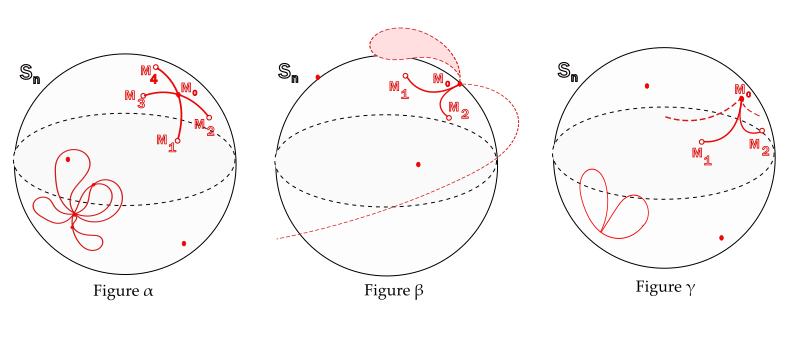}
\caption{Local illustration of possible forms of $\mathcal{Z}(p)\cap \SSS.$ The set $M_0$ represents a component lying in the relative boundary of the components $M_i,$ $i=1,2,3,4.$ Figure $\alpha$ illustrates the case of crossing manifolds which may differ from a union of manifolds in $\SSS.$ Proposition~\ref{prop} says Figure $\beta$ (real analytic extension beyond the unit sphere) cannot occur.  Figure $\gamma$ illustrates a potentially problematic case.}
\label{figure1}
\end{figure*}

Theorem~\ref{peak contained locally} together with Theorem~\ref{ThmF} give a characterization of peak sets: $K\subset\SSS$ is a  peak set for $A^\infty(\BB)$ if and only if it is contained locally in $C^\infty$ complex tangential manifolds lying in $\SSS.$ Let us also not that the condition $K,F\subset \SSS$ peak sets for $A^\infty(\BB)$ does not always imply that $K\cup F$ is a peak set for $A^\infty(\BB).$ Consider the counterexample from \cite{Sibony}: the polynomials $\frac{1}{2}(z^2+w^2+1)$ and $\frac{1}{2}(z^2-w^2+1)$ in $\Ball$ peak on the the sets $\{(\cos\theta,\sin \theta):\theta\in [-\pi,\pi]\}$ and $\{(\cos\theta,i\sin \theta):\theta\in [-\pi,\pi]\},$ respectively, whereas their union cannot be a peak set because violates the characterization of peak sets at the intersection points $(1,0)$ and $(1,0)$.

The second crucial result concerns Nash submanifolds and semi-algebraic sets, and it is known as Wing Lemma; a generalization of the Nash Curve Selection Lemma which is a semi-algebraic consequence of the statements formulated in \cite{Cartan} by F. Bruhat, H. Cartan and \cite{Milnor} by J. Milnor. See also \cite{Kurdyka1} for the injective counterpart of the Nash Curve Selection Lemma.

\begin{theorem}[{\cite[Theorem 9.7.10]{Algebraic}} and \cite{Robson}]\label{Wing lemma}
Let $Y\subset \RR^{2n}$ be Nash submanifold and $S\subset \RR^{2n}$ a semi-algebraic set such that $Y\cap S=\emptyset$ and $Y\subset \overline{S}.$ There exists a finite number of Nash wings
$$w_i:(-1,1)\times U_i\rightarrow \RR^{2n}, \quad i=1,...,k,$$
where $U_1,...,U_k$ are open semi-algebraic subsets of $Y,$ such that $w_i((0,1)\times U_i)$ is contained in $S,$ for $i=1,...,k,$ and $\dim(Y\setminus\cup_{i=1}^{k}U_i)<\dim(Y).$
\end{theorem}
    
With these tools in hands, we may obtain explicit descriptions of $\mathcal{Z}(p)\cap \SSS$ in terms of interpolation sets. 

\section{Main Results}\label{sec: Main results}
Set $\alpha=(0',1)=(0,0,...,1)\in \SSS.$ 

\begin{proposition}\label{prop}
Let $p\in\CC[z_1,\dots,z_n]$ be a stable polynomial and $\mathcal{Z}(p)\cap \SSS=\cup_iM_i.$ If $\alpha\in \overline{M_i}\setminus M_i,$ for some intex $i,$ then there exist $\epsilon>0$ and a real analytic curve $\gamma_\alpha:(-\epsilon,\epsilon)\rightarrow \mathcal{Z}(p)\cap \SSS$ which is homeomorphism onto its image and enjoys the following properties:
\\[4pt]
$(i)$ $\gamma_\alpha((0,\epsilon))\subset M_i,$ $\gamma_\alpha(0)=\alpha;$\\[3pt]
$(ii)$ $\gamma_\alpha'(t)\neq 0,$ for all $t\neq0;$\\[3pt]
$(iii)$ there exists unique $k$ such that $\gamma((-\epsilon,0))\subset M_k$;\\[3pt]
$(iv)$ there exists $\delta_0>0$ such that $\gamma((-\epsilon,\epsilon))\cap \Ba(0,\delta)$ has exactly two components meeting at $\alpha,$ for all $\delta<\delta_0;$\\[3pt]
$(v)$ $\gamma_\alpha$ is complex tangential;\\[3pt]
$(vi)$ there exist $\nu>1,$ $b_\nu^1,...,b_\nu^{n-1},a\in \CC,$ such that
$$|b_\nu^1|+...+|b_\nu^{n-1}|>0, \quad a=-\frac{|b_\nu^1|^2+...+|b_\nu^{n-1}|^2}{2},$$
and 
$$\gamma_\alpha(t)=\Big(b^1_\nu t^\nu+o(t^\nu), ...,b^{n-1}_\nu t^\nu+o(t^\nu), 1 +a t^{2\nu}+o(t^{2\nu})\Big),$$
for all $t\in (-\epsilon,\epsilon)$
\end{proposition}

The proposition above locally describes the $1$-dimensional boundary zeros of stable polynomials. Let us note that for sake of simplicity we have chosen $\alpha=(0',1)$ to be a boundary point of the component $M_i.$ The general case $\alpha\in \overline{M_i}\setminus M_i$ may be achieved by utilizing unitary matrices. Of course, the curve $\gamma_\alpha$ might not be unique. See also \cite{Knese}, where the authors follow a different approach.

Recall that a Nash stratification of a semi-algebraic set $A\subset\RR^{2n}$ is a finite partition $\{A_i\}_i$ for $A,$ where each $A_i$ is a semi-algebraically connected Nash submanifold of $\RR^{2n},$ such that, if $A_i\cap \overline{A_j}\neq \emptyset$ and $i\neq j,$ then $A_i\subset \overline{A_j}.$ The components $A_i$ are called Nash strata, see \cite[Definition 9.1.7]{Algebraic}. 

If $\dim_\RR \mathcal{Z}(p)\cap \SSS=1,$ then the finite decomposition $\{M_i\}_i$ in Theorem~\ref{zeroset} may be chosen to form additionally a Nash stratification. Indeed, consider the decomposition $\mathcal{Z}(p)\cap\SSS=\cup_i M_i$ in Theorem~\ref{zeroset}. Note that $\overline{M_i}\setminus M_i\neq \emptyset$ and $\overline{M_i}\setminus M_i$ is a semi-algebraic set with $\dim_\RR (\overline{M_i}\setminus M_i)<1.$ Thus, $\overline{M_i}\setminus M_i$ is a finite set. For each $\alpha\in \overline{M_i}\setminus M_i,$ pick a curve $\gamma_\alpha:(-1,1)\rightarrow\RR^{2n}$ from Proposition~\ref{prop}. Since $M_i$ is an one dimensional connected Nash submanifold, we deduce that the boundary has at most two distinct points. If it has one point, namely a loop without a point, then we may split it as follows: $M_i=\phi_i((1,0))\cup\{\phi_i(0)\}\cup\phi_i((0,1)),$ where each one dimensional component has exactly two distinct points in the relative boundary.

Let $I_1$ the set of all indexes $i$ such that $\overline{M_1}\cap M_i\neq \emptyset$ and let $i\in I_1\neq\emptyset.$ If $\alpha\in M_1,$ then $M_1\cap M_i\neq \emptyset$ which is impossible. Thus, $\alpha\in \overline{M_1}\setminus M_1.$ Note that the case $M_1=\{\alpha\}$ cannot occur. Whence, there exist two distinct points $\alpha,\beta\in \overline{M_1}\setminus M_1$. Let us suppose that $\alpha,\beta \in M_i.$ Set $\alpha:=\phi_i(t^i_\alpha)$ and $\beta:=\phi_i(t^i_\beta),$ where $M_i=\phi_i((-1,1)),$ and $t^i_\alpha<t^i_\beta,$ $t^i_\alpha,t^i_\beta\in (-1,1).$ Then we may decompose the component $M_i$ as follows: $M_i=\phi_i((-1,t^i_\alpha))\cup\{\phi_i(t^i_\alpha)\}\cup\phi_i((t^i_\alpha,t^i_\beta))\cup\{\phi_i(t^i_\beta)\}\cup \phi_i((t^i_\beta,1)).$ We obtain an analogue decomposition for the simpler case $\alpha\in M_i$ and $\beta\notin M_i.$ Next, if $j\in I_1,$ $j\neq i,$ then the only possible case is $\alpha,\beta\in\overline{M_1}\setminus M_1,$ $\alpha\neq \beta,$ and $\alpha\in M_i,$ $\beta\in M_j.$ Hence, $\mathrm{card}(I_1)\leq2.$ In the case $\mathrm{card}(I_1)=2$ and $\beta\in M_j$, we obtain the decomposition $M_j=\phi_j((-1,t^j_\beta))\cup \{\phi_j(t^j_\beta)\}\cup \phi_j((t_\beta^j,1)).$ Continuing this process inductively on the finite partition $\{M_i\}_i$, one may obtain a Nash stratification, namely,  $\mathcal{Z}(p)\cap \SSS=\cup_iN_i,$ where $N_i$ is a Nash submanifold lying in $\SSS$ and Nash diffeomorphic to  $(-1,1)^{\dim_\RR N_i},$ $\dim_\RR N_i=0,1.$ 

Henceforth, a \emph{Nash stratification} of $\mathcal{Z}(p)\cap \SSS$ refers to a finite partition $\{M_i\}_i$ of $M,$ where $M_i\cap \overline{M_j}\neq \emptyset,$ $i\neq j,$ yields $M_i\subset \overline{M_j},$ and each component $M_i$ is a Nash submanifold lying in $\SSS$ which is diffeomorphic to $(-1,1)^{\dim_\RR M_i},$ $\dim_\RR M_i=0,1,$ by a Nash diffeomorphism that is complex tangential. The relative boundary $\overline{M_i}\setminus M_i$ of each one dimensional component contains exactly two distinct points. The Nash stratification is not unique.

Let us note that in the property $(i)$ of Proposition~\ref{prop}, the index $k$ might be equal to $i$ for $\dim_\RR(M_i)\geq 2.$ If $\dim_\RR\mathcal{Z}(p)\cap \SSS=1$ and $\{M_i\}_i$ forms a Nash stratification, then $k\neq i.$

It is possible that the map $\gamma_\alpha:(-\epsilon,\epsilon)\rightarrow\mathcal{Z}(p)\cap\SSS$ enjoying the properties of Proposition~\ref{prop} yields a $C^\infty$ submanifold lying in $\mathcal{Z}(p)\cap \SSS$, namely, $\gamma_\alpha((-\epsilon,\epsilon))\subset \mathcal{Z}(p)\cap \SSS$ is a $C^\infty$ manifold in $\SSS.$ For instance, for any $\nu=3,5,...,$ the curve $\gamma_{(0,1)}(t)=(t^{\nu},\sqrt{1-t^{2\nu}})$ around zero  yields a complex tangential real analytic submanifold in $\mathbb{S}_2$ although the map is not a diffeomorphism itself (we talk about a homeomorphism on an open interval, hence, it cannot yield a manifold with boundary like the case $\nu=2$). In particular, it is contained in $\mathcal{Z}(1-z^2-w^2)\cap \SSS.$ 

Problematic case is considered to be whenever the closure of an  one dimensional component $\overline{M_i}$ obtained from any Nash stratification  does not lie in one dimensional $C^\infty$ manifolds in $\SSS,$ see Figure 1 ($\beta$ and $\gamma$) for possible illustrations. To work out this problem, let us consider a certain example in algebraic geometry in $\RR^2.$ Let $\gamma(t):=(t^2,t^3),$ $t\in [0,\epsilon),$ and assume that there exists an one dimensional smooth submanifold $\Gamma\subset \RR^2$ such that $\gamma([0,\epsilon))\subset \Gamma.$ Since $\Gamma$ is smooth and $(0,0)\in\Gamma,$ there exist a smooth $\omega:(-\delta,\delta)\rightarrow \Gamma$ such that $\omega(0)=(0,0)$ and $\omega'(0)\neq 0.$ If $\omega(t)=(\omega_1(t),\omega_2(t)),$ then $\omega_1(t)^2=\omega_2(t)^3,$ for all $t\geq 0$ close to $0.$ Differentiating three times, we obtain $\omega_1'(0)=\omega'_2(0)=0,$ which is a contradiction. 

On the other hand, let $\alpha=(0',1)$ and let $\gamma:=\gamma_\alpha=(\gamma_1,..,\gamma_{n-1},\gamma_n)$ be the Nash curve obtained from Proposition~\ref{prop}. Then there exists $\delta>0$ and a $C^1$ complex tangential submanifold $M\subset\SSS$ such that $\gamma([0,\delta))\subset M.$ Indeed, let $\tilde{\gamma}:(0,\epsilon)\rightarrow\CC^n$ be the Nash curve defined by $\tilde{\gamma}=(-\gamma_1,...,-\gamma_{n-1},\gamma_n).$ The nature of $\gamma$ yields that $\gamma((0,\epsilon))\cap\tilde{\gamma}((0,\epsilon))$ has finitely many connected components and near $\alpha$ is empty since $\gamma_i,$ $i=1,...,n-1,$ are not identically zero, thus there exists $\delta>0$ such that $M:=\tilde{\gamma}((0,\delta))\cup\{\alpha\}\cup\gamma((0,\delta))$ is a union of two disjoint real analytic submanifolds (with exactly two points on the boundary) sticking at the unique point $\alpha.$ To show that $M$ is a $C^1$ complex tangential manifold lying in the unit sphere, it is enough to find a $C^1$ chart around the point $\alpha.$ Let $0<\theta<\delta$ and set $\omega:(-\theta,\theta)\rightarrow \CC^n$ by 
$$\omega(t)=(b^1_\nu t+b^1_{\nu+1}t^{1+1/\nu}+..,...,b_{\nu}^{n-1}t+b_{\nu+1}^nt^{1+1/\nu}+..,1+a_{2\nu}t^{2}+a_{2\nu+1}t^{2+1/\nu}+..), \quad t>0,$$
$$\omega(t)=(b^1_\nu t-b^1_{\nu+1}|t|^{1+1/\nu}-..,...,b_{\nu}^{n-1}t-b_{\nu+1}^n|t|^{1+1/\nu}-..,1+a_{2\nu}|t|^{2}+a_{2\nu+1}|t|^{2+1/\nu}+..), \quad t<0,$$
and $\omega(0)=\alpha.$ Then $\lim_{t\rightarrow 0^-}\omega'(t)=(b_\nu^1,...,b_{\nu}^{n-1},0)=\lim_{t\rightarrow 0^-}\omega'(t)\neq 0,$
and $\omega(t)\neq 0,$ for all $t\neq 0.$ Furthermore, $\omega((-\theta,\theta))\subset M$ and $\omega(0)=\alpha\in M,$ ergo $M$ is an one-dimensional $C^1$ complex tangential submanifold lying on $\SSS.$

Hence, we need to check out under which conditions the boundary zeros of a polynomial does not contain such problematic points; either points where submanifolds lying in the boundary zeros cross each other or the union of their closures build a cusp. One might strongly claim in full generality that there exist stable polynomials and  $\alpha\in \overline{M_i}\setminus M_i$ taken from their boundary zeros such that for any $\gamma_\alpha,$ the set $\gamma_\alpha([0,\epsilon))$ is not contained in any one dimensional $C^\infty$ manifold in $\SSS.$ See Figure~\ref{figure1} for possible illustrations of the boundary zeros.

See crucial theory and examples regarding the geometry of singularities in \cite{Bickel}, \cite{Sola1} and \cite{Knese}. See also \cite{Mostowski} about the behaviour of a germ of an analytic curve close to its boundary.

All in all, it is reasonable to formulate the following definitions.

\begin{definition}
Let $p\in \CC[z_1,...,z_n]$ be a stable polynomial and let $\mathcal{Z}(p)\cap \SSS=\cup_iM_i.$ We say that a point $\alpha\in \overline{M_i}\setminus M_i$ is regular if there exists a map $\gamma_\alpha:(-\epsilon,\epsilon)\rightarrow\mathcal{Z}(p)\cap\SSS$ that enjoys the properties of Proposition~\ref{prop} and a $C^\infty$ complex tangential submanifold $N$ lying in $\SSS$ such that $\gamma_\alpha([0,\epsilon))\subset N.$ We say that the point $\alpha$ is strong regular if $N$ is an one dimensional complex tangential $C^\infty$ submanifold.
\end{definition}

Let us note that the condition $\gamma_\alpha([0,\epsilon))\subset N$ for a $C^\infty$ complex tangential submanifold is more general than assuming $\dim_\RR N=1;$ which is the case in the $\Ball$ setting. For $n\geq 3,$ this definition allows the case where $\gamma_\alpha([0,\epsilon))$ is not contained in any one dimensional smooth complex tangential manifold whereas $\alpha$ is not a problematic point in terms of interpolations sets.  Hence, we may specify the notion of regularities (resp. singularities) for the case $\dim_\RR(\mathcal{Z}(p)\cap \SSS)=1.$

\begin{definition}\label{def}
Let $p\in \CC[z_1,...,z_n]$ be a stable polynomial and assume that $\dim_\RR(\mathcal{Z}(p)\cap \SSS)=1.$ Pick a Nash stratification $\mathcal{Z}(p)\cap \SSS=\cup_i M_i.$ We say that $\mathcal{Z}(p)\cap\SSS$ has an essential singularity at the accumulation point $\alpha\in \overline{M_i}\setminus M_i$ if $\alpha$ is not regular with respect to $M_i.$ We say that $\mathcal{Z}(p)\cap\SSS$ has a weak essential singularity at the accumulation point $\alpha\in \overline{M_i}\setminus M_i$ if $\alpha$ is not strong regular with respect to $M_i.$ Furthermore, we say that $\mathcal{Z}(p)\cap\SSS$ has no essential singularities (resp. weak) if it has no essential singularity (resp. weak) at any $\alpha\in \overline{M_i}\setminus M_i,$ for all $i.$ 
\end{definition}

We shall deal with the question of whether polynomials possessing weak essential singularities exist. For the setting $n\geq 3$ such polynomials indeed exist. The idea is to pick a polynomial that is a sum of suitable polynomials. Consider the polynomial 
$$\pi(z):=1-z_1^2-z_2^2-z_3^2+1/13(z_2^2-z_1^3)^2,$$
$z=(z_1,z_2,z_3).$ Set $z=x+iy,$ where $z_i=x_i+iy_i,$ and $\sigma(z_1,z_2):=z_2^3-z_1^2.$ Then $\textrm{Im }\sigma(z_1,z_2)=2x_2y_2+y_1(y_1^2-3x_1^2).$ Note that $||z_i||\leq1,$ for $i=1,2,3.$ By  $x_1^2+y_1^2\leq 1,$ we obtain $|y_1^2-3x_1^2|\leq 3.$ Furthermore, Cauchy-Schwarz inequality yields $|\textrm{Im }\sigma(z_1,z_2)|\leq \sqrt{13(|y_1|^2+|y_2|^2)}.$ We have 
\begin{align*}
\textrm{Re }\pi(z)&=1-||x||^2+||y||^2+1/13[(\textrm{Re }\sigma(z))^2-(\textrm{Im }\sigma(z))^2]
\geq 1-||x||^2+1/13(\textrm{Re }\sigma(z))^2
\end{align*}
which is positive for $||z||<1.$ This proves that $\pi$ is stable. In particular, if $||z||=1$ and $\pi(z)=0,$ then $||x||=1,$ $y=0$ and $\textrm{Re }\sigma(z)=0,$ from which follows that $1-z_1^2-z_2^2-z_3^2=0$ and $\sigma(z)=0.$ Whence, $\mathcal{Z}(\pi)\cap \mathbb S_3\subset \{z\in \CC^3:\textrm{Im }z=0,||z||=1,\sigma(z)=0\}$ which is $1$-dimensional near to the point $\alpha:=(0,0,1).$ Let $\gamma(t):=(t^2,t^3,\sqrt{1-t^4-t^6}),$ for $t$ close to $0.$ Note that $\gamma$ is a Nash curve, $||\gamma(t)||=1,$ $\gamma(0)=\alpha,$ and $\pi(\gamma(t))=0,$ for all $t$ close to $0.$  Last, for any $\epsilon>0$ the portion $\gamma([0,\epsilon))$ is not contained in any one dimensional $C^\infty$ submanifold lying in $\mathbb {S}_3,$ hence, $\alpha$ is a weak essential singularity for a component of  $\mathcal{Z}(\pi)\cap \mathbb S_3.$ On the other hand, $\alpha$ is not an essential singularity as $\gamma_\alpha([0,\epsilon))$ lies on a $2$-dimensional sphere. The shape of $\mathcal{Z}(\pi)\cap \mathbb S_3$ looks like the boundary of a nicely cut grapefruit rind. The above definition on weak essential singularities and this example suggest that the covering framework is inherently insufficient to fully characterize the cyclicity of any stable polynomial in the setting $n\geq 3$. For the case $n=2$, one might reasonably assert that such singularities do not exist, which would completely characterize the boundary zeros of stable polynomials. The idea is to apply theory coming up either in \cite{KV} concerning the Weierstrass preparation theorem and Puiseux's Parametrization theorem or in \cite{Knese} (moving the problem from the sphere by a biholomorphic change of coordinates to another setting).  We leave this open problem for near future study.

In light of the Proposition~\ref{prop} and the Definition~\ref{def}, we may answer to the \hyperref[Question B]{Question (B)} and give an explicit description of $\mathcal{Z}(p)\cap\SSS$. We reach a characterization of the boundary zeros in terms of interpolation sets.

\begin{theorem}\label{Character}
Let $p\in\CC[z_1,\dots,z_n]$ be a stable polynomial such that $\dim_\RR\mathcal{Z}(p)\cap\SSS=1$ and consider a Nash stratification $\mathcal{Z}(p)\cap \SSS=\cup_iM_i.$ The set $\mathcal{Z}(p)\cap\SSS$ has no essential singularities if and only if each $\overline{M_i}$ is a peak set for $A^\infty(\BB).$
\end{theorem}

In cyclicity problem of polynomials, we shall make use of the theorem below.

\begin{theorem}\label{theoremm}
Let $p\in\CC[z_1,\dots,z_n]$ be a stable polynomial such that $\dim_\RR\mathcal{Z}(p)\cap\SSS=1$ and consider a Nash stratification $\mathcal{Z}(p)\cap \SSS=\cup_iM_i.$ If $\mathcal{Z}(p)\cap \SSS$ has no weak essential singularities, then there exists finitely many peak sets $K_j$ for $A^\infty(\BB)$ such that $\mathcal{Z}(p)\cap \SSS=\cup_jK_j.$ Each one of them is contained in an one dimensional complex tangential $C^\infty$ submanifold in $\SSS.$
\end{theorem}

The proposition below examines the case where the zero set of a stable polynomial coincide locally (points on the unit sphere) with the graph of a holomorphic function; we state it in terms of its Weierstrass form. Let us denote the reduced variable by $z':=(z_1,...,z_{n-1}).$ 

\begin{proposition}\label{prop graph}
Let $p\in\CC[z_1,\dots,z_n]$ be a stable polynomial with $\mathcal{Z}(p)\cap \SSS=\cup_iM_i.$ Assume that for each  $\alpha\in \mathcal{Z}(p)\cap \SSS,$ there exists a polydisk $\mathrm{P}(\alpha,\delta)$ such that $p(z)=h(z)(z_n-a(z')),$ where $h\in \mathrm{Hol}(\mathrm{P}(\alpha,\delta))$ is zero-free and $a\in\mathrm{Hol}(\Pol(\alpha',\delta)).$ Then $\overline{M_i}$ is contained locally in real analytic totally real submanifolds lying $\SSS.$ Furthermore, if $\dim_\RR M_i=n-1,$ then $\overline{M_i}$ is a peak set for $A^\infty(\BB).$
\end{proposition}   

Let us note that, if in the above statement we additionally assume that for each $i,$ for each $\alpha\in \overline{M_i},$ and for each open neighbourhood $U$ of $\alpha,$ we have $\dim_\RR M_i\cap U=n-1,$ then Proposition~\ref{prop graph} becomes a characterization of $\mathcal{Z}(p)\cap \SSS$.

For $n=2,$ we may say more things due to the maximal dimension. Proposition~\ref{prop graph} together with Theorem~\ref{Character} immediately yield a characterization of the boundary zeros of stable polynomials whose the local zero set around points on the unit sphere is the graph of a holomorphic function, see also \cite{Knese}.

\begin{theorem}\label{n=2 characterization zero set graph}
Let $p\in\CC[z_1,z_2]$ be a stable polynomial with Nash stratification $\mathcal{Z}(p)\cap \mathbb S_2=\cup_iM_i.$ Assume that for each  $\alpha\in \mathcal{Z}(p)\cap \mathbb S_2,$ there exists a polydisk $\mathrm{P}(\alpha,\delta)$ such that $p(z)=h(z)(z_2-a(z_1)),$ where $h\in \mathrm{Hol}(\mathrm{P}(\alpha,\delta))$ is zero-free and $a\in\mathrm{Hol}(\Pol(\alpha',\delta)).$ Then $\mathcal{Z}(p)\cap \mathbb S_2$ has no weak essential singularities and each $\overline{M_i}$ is a peak set for $A^\infty(\BB)$.
\end{theorem}

Turning to the case $\dim_\RR\mathcal{Z}(p)\cap \SSS\geq2,$ we obtain a countable covering satisfying probably the best properties that could be expected. The difficulty is that it is not clear for the moment whether such coverings can be applied to the cyclicity problem. 

Recall that a semi-algebraic set $A\subset \RR^N$ is called semi-algebraically path connected if for each $x,y\in A,$ there exists a semi-algebraic mapping $\varphi:[0,1]\rightarrow A$ such that $\varphi(0)=x$ and $\varphi(1)=y.$ In particular, a semi-algebraic set $A\subset \RR^N$ is  semi-algebraically path connected if and only if it is connected, see \cite[Theorem 2.4.5, Proposition 2.5.13]{Algebraic}.

\begin{theorem}\label{path conn peak sets}
Let $p\in\CC[z_1,\dots,z_n]$ be a stable polynomial with $\mathcal{Z}(p)\cap \SSS=\cup_iM_i.$ Assume that $\dim_\RR M_i\geq 2,$ the boundary of the component $M_i$ contains one and only one point which is strong regular and there exists $\delta^*>0$ such that $(\CC^n\setminus \Ba(\alpha,\delta))\cap M_i$ is connected for all $\delta<\delta^*$. Then $\overline{M_i}=\cup_j K_j,$ where $\alpha\in K_j\subset \overline{M_i}$ form an increasing sequence of semi-algebraically connected compact peak sets for $A^\infty(\BB).$ In particular, each $K_j$ is contained in two distinct $C^\infty$ complex tangential submanifolds $N_j$ and $M_j$ such that $\dim_\RR N_j=1$ and $\dim_\RR M_j=\dim_\RR M_i.$
\end{theorem}

On the other half, we obtain the following.

\begin{theorem}\label{path conn peak sets2}
Let $p\in\CC[z_1,\dots,z_n]$ be a stable polynomial with $\mathcal{Z}(p)\cap \SSS=\cup_iM_i.$ Assume that $\dim_\RR M_i\geq 2,$ for all $i,$ and the relative boundary of each component $M_i$ contains one and only one point. If there exists a semi-algebraically path connected set  $K\subset \overline{M_i},$ which is a compact peak sets for $A^\infty(\BB)$ and $K\cap M_i\neq \emptyset,$ then $\alpha\in \overline{M_i}\setminus M_i$ is a regular point.
\end{theorem}

The above established theory leads to a characterization of cyclicity for polynomials having no weak essential singularities on the Dirichlet-type space $\DDD_{n-\frac{1}{2}}(\BB),$ see \cite{VZ} for the  full conjectured characterization.

\begin{theorem}\label{main2}
Let $p\in\CC[z_1,\dots,z_n]$ be a stable polynomial. Assume that $\dim_\RR\mathcal{Z}(p)\cap \SSS=1$ and it has no weak essential singularities. Then $p$ is cyclic in $\DDD_{\alpha}(\BB)$ if and only if $\alpha\leq n-\frac{1}{2}.$
\end{theorem}

Last, following the procedure of the proofs of Proposition~\ref{prop graph}, Theorem~\ref{main2} and a compactness argument, we obtain a  characterization of cyclicity for certain polynomials. 

\begin{theorem}\label{cyclicity graph}
Let $p\in\CC[z_1,\dots,z_n]$ be a stable polynomial with $\mathcal{Z}(p)\cap \SSS=\cup_iM_i$ and $\dim_\RR \mathcal{Z}(p)\cap \SSS=n-1.$ Assume that for each  $\alpha\in \mathcal{Z}(p)\cap \SSS,$ there exists a polydisk $\mathrm{P}(\alpha,\delta)$ such that $p(z)=h(z)(z_n-a(z')),$ where $h\in \mathrm{Hol}(\mathrm{P}(\alpha,\delta))$ is zero-free and $a\in\mathrm{Hol}(\Pol(\alpha',\delta)).$ Moreover, assume that each component $M_i$ which is not an isolated point with respect to $\mathcal{Z}(p)\cap \SSS$, is contained in some $\overline{M_{j}},$ where $ \dim_\RR M_j=n-1.$ Then $p$ is cyclic in $\DDD_\alpha(\BB)$ if and only if $\alpha\leq \frac{n+1}{2}.$  
\end{theorem}

\section{Proofs}\label{sec:proofs}
\subsection{Proof of Proposition~\ref{prop}}
\begin{proof}
Let $\alpha\in \overline{M_i}\setminus M_i$. Applying Wing Lemma for the semi-algebraic set $S=M_i$ and $Y=\{\alpha\}$ (see also \cite[0.2. Lemme]{Kurdyka1}), we deduce that there exists a Nash wing $\gamma:=\gamma_\alpha:(-1,1)\rightarrow \CC^n$  which is homeomorphism onto its image and enjoys the following properties: $\gamma((0,1))\subset M_i;$ $\gamma(0)=\alpha=(0',1);$ $\gamma'(t)\neq 0,$ for all $t\neq 0;$ there exists $\nu=2,3,..,$ such that $\gamma^{(j)}(0)=0,$ for all $j=1,2,...,\nu-1,$ and $\gamma^{(\nu)}(0)\neq 0.$

Identity principle ensures that $\gamma_\alpha((-1,1))\subset \mathcal{Z}(p)\cap \SSS.$

Let $0<\epsilon<1$ and consider the restriction $\gamma:[-\epsilon,\epsilon]\rightarrow\SSS.$  Note that $\gamma([-\epsilon,\epsilon])\subset\mathcal{Z}(p)\cap \SSS,$ namely, it is contained in a (Z)-set, and it is also compact. Whence, $\gamma([-\epsilon,\epsilon])$ is a (Z)-set too, see Remark~\ref{equivalence-Z-PI}. Thus, $\gamma([-\epsilon,\epsilon])$ is a (PI)-set, see Remark~\ref{equivalence-Z-PI}, that is, $\gamma:[-\epsilon,\epsilon]\rightarrow \mathcal{Z}(p)\cap\SSS$ is a complex tangential curve, see \cite[Theorems 10.5.4, 11.2.5]{Rudin}.

Recall that $\gamma$ is a Nash homeomorphism, thereby yielding a connected Nash curve without self intersections. The germs $\gamma((-\epsilon,0))$ and $\gamma((0,\epsilon))$ are connected semi-algebraic sets lying in $\mathcal{Z}(p)\cap \SSS$ and $\gamma((-\epsilon,0))\cap \gamma((0,\epsilon))=\emptyset.$ For any $k,$ the set $\gamma((-\epsilon,0))\cap M_k$ is also a semi-algebraic set, and thus, it admits a finite decomposition of semi-algebraic connected components.  Whence, shrinking $\epsilon>0$ if necessary, there exists unique $k$ such that $\gamma((-\epsilon,0))\subset M_k$ (it might be $k=i$).

We follow the Remark 9.5.8 (Subsections 9.3 and 9.5) of \cite{Algebraic} concerning the half branches and local germs of a Nash curve. Shrinking $\epsilon>0$ if necessary, there exists $\delta_0>0$ such that $[\gamma([-\epsilon,\epsilon])\cap \Ba(\alpha,\delta)]\setminus\{\alpha\}=\cup_i(C_i\cap\Ba(\alpha,\delta)),$ for all $\delta<\delta_0$, where $C_i$ and $C_i\cap\Ba(\alpha,\delta)$ are pairwise disjoint finitely many semi-algebraically connected germs meeting at the point $\alpha.$ The germs $C_i$ do not depend on $\delta<\delta_0.$ In particular, $C_i=\gamma((t_i,s_i)),$ for fixed $t_i,s_i\in [-\epsilon,\epsilon].$ The closures of the components meet at the point $\alpha$ and the Nash function is a homeomorphism, hence, at least one of $t_i$ and $s_i$ should be in $\{0,-\epsilon,\epsilon\}.$ Since $C_i$ are connected and pairwise disjoint, it turns out that there exist at most four components.  Therefore, there exists $\epsilon>0$ such that $\gamma((-\epsilon,\epsilon))\cap \Ba(0,\delta)$ has exactly two connected components $\gamma((-\epsilon,0))\cap \Ba(0,\delta)$ and $\gamma((0,\epsilon))\cap \Ba(0,\delta)$ meeting at $\alpha,$ for all $\delta<\delta_0.$ 

Since $\gamma$ is a real analytic function and $\gamma(0)=(0',1)$, we may assume that its coordinates may be expanded as follows
$$\gamma_i(t)=\sum_{k=1}^{+\infty}b^i_kt^k, \quad i=1,2,...,n-1,$$
and
$$\gamma_n(t)=1+\sum_{k=0}^{+\infty}a_kt^k,$$
for all $t\in (-\epsilon,\epsilon).$

The fact that $\gamma^{(j)}(0)=0,$ for all $j=1,2,...,\nu-1,$ and $\gamma^{(\nu)}(0)\neq 0,$ where $\nu=2,3,...$ is fixed, yield
$$\gamma_i(t)=b^i_\nu t^\nu+o(t^\nu), \quad i=1,...,n-1,$$
and 
$$\gamma_{n}(t)=1+a_\nu t^\nu+...+a_{2\nu-1}t^{2\nu-1}+a_{2\nu}t^{2\nu}+o(t^{2\nu}),$$
where
$$|b_\nu^1|+...+|b_\nu^{n-1}|+|a_\nu|>0.$$

Furthermore,
\begin{align*}
\gamma_i'(t)\overline{\gamma_i(t)}&=[\nu b_\nu^it^{\nu-1}+o(t^{\nu-1})][\overline{b_\nu^i}t^\nu+o(t^\nu)]=\nu|b_\nu^i|^2t^{2\nu-1}+o(t^{2\nu-1}),
\end{align*}
for all $i=1,...,n-1,$ and
\begin{align*}
\gamma_n'(t)\overline{\gamma_n(t)}&=[\nu a_{\nu}t^{\nu-1}+...+(2\nu-1)a_{2\nu-1}t^{2\nu-2}+2\nu a_{2\nu}t^{2\nu-1}+o(t^{2\nu-1})][1+\overline{a_\nu}t^\nu+o(t^\nu)]\\
&=\nu a_{\nu}t^{\nu-1}+...+(2\nu-1)a_{2\nu-1}t^{2\nu-2}+2\nu a_{2\nu}t^{2\nu-1}+\nu a_{\nu-1}\overline{a_\nu}t^{2\nu-1}+o(t^{2\nu-1}).
\end{align*}
Hence, by the complex-tangential condition $\langle\gamma'(t),\gamma(t)\rangle=0,$ for all $t\in [-\epsilon,\epsilon],$ we obtain
\begin{align*}
\sum_{i=1}^{n-1}&\nu|b_\nu^i|^2t^{2\nu-1}+\\
&+\nu a_{\nu}t^{\nu-1}+...+(2\nu-1)a_{2\nu-1}t^{2\nu-2}+2\nu a_{2\nu}t^{2\nu-1}+\nu a_{\nu-1}\overline{a_\nu}t^{2\nu-1}+o(t^{2\nu-1})=0.
\end{align*}
It follows that
$$a_\nu=...=a_{2\nu-1}=0.$$
Whence,
$$ \sum_{i=1}^{n-1}\nu|b_\nu^i|^2t^{2\nu-1}+2\nu a_{2\nu}t^{2\nu-1}+o(t^{2\nu-1})=0,$$
and thus
$$a_{2\nu}=-\frac{\sum_{i=1}^{n-1}|b_\nu^i|^2}{2}<0,$$
since 
$$|b_\nu^1|+...+|b_\nu^{n-1}|+|a_\nu|=|b_\nu^1|+...+|b_\nu^{n-1}|>0.$$
Hence, the real analytic curve is classified as follows
$$\gamma(t)=\Big(b^1_\nu t^\nu+o(t^\nu), ...,b^{n-1}_\nu t^\nu+o(t^\nu), 1 -\frac{|b_\nu^1|^2+...+|b_\nu^{n-1}|^2}{2}t^{2\nu}+o(t^{2\nu})\Big),$$
for all $t\in (-\epsilon,\epsilon)$ and the claims are complete.
\end{proof}

\subsection{Proof of Theorem~\ref{Character}}
\begin{proof}
Let $\mathcal{Z}(p)\cap\SSS=\cup M_i$ be a Nash stratification. If $\overline{M_i}$ is a point, then it is also a peak set for $A^\infty(\BB).$

For the rest of $\mathcal{Z}(p)\cap \SSS,$ choose an one dimensional component $M:=M_i$ The set $M$ has exactly two points in the relative boundary. Let $\alpha\in \overline{M}.$ If $\alpha\in M,$ then there exist an open neighbourhood $U\subset \CC^n$ of $\alpha$ such that $U\cap \overline{M}$ is contained in $M.$ If $\alpha\in \overline{M}\setminus M,$ then by assumption, there exists an open neighbourhood $U\subset \CC^n$ of $\alpha$ such that $\overline{M}\cap U\subset\gamma_\alpha([0,\epsilon)),$ where $\gamma_\alpha([0,\epsilon))$ lies in a complex tangential $C^\infty$ manifold in $\SSS.$ Theorem~\ref{ThmF} shows that $\overline{M}$ is a peak set for $A^\infty(\BB).$

On the other hand, assume that each $\overline{M_i}$ is a peak set for $A^\infty(\BB).$ Suppose in contrary that there exists an essential singularity $\alpha\in \overline{M_i}\setminus M_i.$ Pick a map $\gamma_\alpha$ from Proposition~\ref{prop}. Recall that compact subsets of a peak set for $A^\infty(\BB)$ are also peak sets for $A^\infty(\BB),$ see \cite[Theorem 11]{Chaumat}. Since $\overline{M_i}$ is a peak set for $A^\infty(\BB),$ it follows that $\gamma_\alpha([0,\epsilon])$ is a peak set for $A^\infty(\BB)$ for some $\epsilon>0.$ Theorem~\ref{peak contained locally} yields that there exists a $C^\infty$ complex tangential submanifold $N$ in $\SSS$ and a neighbourhood $U\subset \CC^n$ of $\alpha$ such that $\gamma_\alpha([0,\epsilon])\cap U$ is contained in $N.$  Summing up, there exists $\delta>0$ such that $\gamma_\alpha([0,\delta))\subset N,$ from which follows that $\alpha$ is not an essential singularity for the $M_i.$ This is a contradiction and the assertion follows. 
\end{proof}

\subsection{Proof of Theorem~\ref{theoremm}}
\begin{proof}
Set $K:=\overline{M_i},$ where $K$ has exactly two points in the relative boundary (hence, not a point). Following the same arguments as in the previous proof, there exist $N_\alpha,N_\beta,$ one dimensional complex tangential $C^\infty$ submanifolds lying on $\SSS$ and $\epsilon>0$ such that $K_\alpha:=\gamma_\alpha([0,\epsilon])\subset N_\alpha$ and $K_\beta:=\gamma_\beta([0,\epsilon])\subset N_\beta.$  Both $K_\alpha$ and $K_\beta$ are peak sets for $A^\infty(\BB)$ and they are contained in one dimensional complex tangential $C^\infty$  manifolds.

Furthermore, there exists compact $K_{\alpha,\beta}\subset M_i$ (its relative interior intersects the relative interiors of $K_\alpha$ and $K_\beta$) such that $K= K_\alpha\cup K_{\alpha,\beta}\cup K_\beta.$ The set $K_{\alpha,\beta}$ is peak set for $A^\infty(\BB)$ since $M_i$ is complex tangential.

Lastly, the components $M_i$ that are isolate points compose a union of finitely many peak sets. The components that are non isolated points are contained in the closure of one dimensional components by the notion of Nash stratification. This yields the assertion.
\end{proof}

\subsection{Proof of Proposition~\ref{prop graph}}
\begin{proof}
It is enough to prove the statement for a non-zero dimensional component $M_i$ and locally around $\alpha=(0',1)\in \overline{M_i}\setminus M_i.$ Write $p$ in its Weierstrass form  $p(z)=h(z)(z_n-a(z')),$ where $\mathrm{P}(\alpha,\delta)$ is a polydisk centered at $\alpha,$ $h\in \mathrm{Hol}(\mathrm{P}(\alpha,\delta))$ is zero-free, $\alpha\in\mathrm{Hol}(\Pol(0',\delta))$ and $\alpha(0')=1.$  Then $\mathcal{Z}(p)\cap \SSS\cap \Pol(\alpha,\delta)=\{(z',a(z'))\in \CC^n:z'\in \Pol(0',\delta), ||z'||^2+|a(z')|^2=1\}.$ Consider the following function:
$$\varphi(z'):=||z'||^2+|a(z')|^2-1, \quad z'\in \Pol(0',\delta),$$
and note that it is a non-negative real analytic strictly plurisubharmonic function. Hence, there exists $\delta'>0$ and a ($n-1$)-dimensional, real analytic, totally real submanifold $M\subset \Ba(0',\delta')$ such that $\mathcal{Z}(\varphi)\cap \Ba(0',\delta')\subset M,$ see \cite{HW}. Next, both sets $\mathcal{Z}(\varphi)$ and $M$ are real analytic, thus they admit a decomposition of finitely many connected real analytic submanifolds in $\CC^n.$  There exists $\tilde{\varphi}:\Omega\rightarrow M,$ where $\Omega:=(-1,1)^{n-1},$ real analytic, diffeomorphism onto its image and such that $\tilde{\varphi}(0)=0'.$ Furthermore, the map $f(x):=(\tilde{\varphi}(x),a(\tilde{\varphi}(x)),$ $x\in \Omega$ yields a ($n-1$)-dimensional, real analytic, submanifold in $\CC^n$ and passing through $\alpha.$  We shall show that $N:=f(\Omega)$ is totally real.  Indeed, let us suppose that $w\in T_\po N\cap iT_\po N,$ $\po=f(x),$ $x\in \Omega,$ and $T_\po N=\{(\tilde{\varphi}'(x)h,(a\circ \tilde{\varphi})'(x)h):h\in \RR^{n-1}\}.$ Then there exists $h_1,h_2$ such that $\tilde{\varphi}'(x)h_1=i\tilde{\varphi}'(x)h_2.$ Set $u_1=\tilde{\varphi}'(x)h_1$ and $u_2=\tilde{\varphi}'(x)h_2.$ Since $M$ is totally real and $\tilde{\varphi}$ is diffeomorphism, we obtain $u_1=u_2=0$ and $w=0.$ Thus, $N$ is totally real. Last, if $z=(z',z_n)\in \Pol(\alpha,\delta)\cap M_i,$ then $z'\in \mathcal{Z}(\varphi)$ and $z_n=a(z').$ Thus, for $z$ close enough to $\alpha,$ we obtain that $z\in N.$ 

Next, let us suppose that $M_i$ is a $(n-1)$-dimensional component. Note that $\alpha\in N.$ We may assume that $N$ is also connected. By the arguments above, there exists an open ball $\Ba(\alpha,\epsilon)$ such that $M_i\cap\Ba(\alpha,\epsilon)\subset N.$ Note that $\dim_\RR (M_i\cap\Ba(\alpha,\epsilon))=n-1.$ We apply the decomposition of Proposition 2.9.10 of \cite{Algebraic} on the semi-algebraic set $M_i\cap\Ba(\alpha,\epsilon)$, thereby obtaining  $M_i\cap\Ba(\alpha,\epsilon)=\cup_j U_j,$ where $U_j$ are pairwise disjoint finitely many Nash submanifolds, each Nash diffeomorphic to an open hypercube $(-1,1)^m,$ $m=0,1,...,n-1.$ Note that $\alpha\in \overline{M_i\cap\Ba(\alpha,\epsilon)}=\cup_j \overline{U_j}.$ Let $J$ be the set of all indexes such that $\alpha\in \overline{U_j}.$ Since $\overline{U_j}$ are finitely many compact sets, we may shrink $\epsilon>0$ such that $M_i\cap\Ba(\alpha,\epsilon)=\cup_{j\in J}( U_j\cap \Ba(\alpha,\epsilon)).$ If $\dim_\RR U_j<n-1,$ for all $j\in J,$ then $\dim_\RR (\cup_{j\in J}(U_j\cap \Ba(\alpha,\epsilon)))=\dim_\RR (M_i\cap\Ba(\alpha,\epsilon))<n-1,$ which contradicts the assumption. Whence, there exist at least one $j$ such that $\alpha\in \overline{U_j}$ and $\dim_\RR U_j=n-1.$ Thus, we may assume that $U_j\subset N,$ and $U_j$ is open in $N.$ In particular, $U_j\subset \mathcal{Z}(p)\cap \SSS,$ and hence, the identity principle leads to $N\subset \mathcal{Z}(p)\cap \SSS.$ Summing up, $N$ is a complex tangential submanifold in $\SSS$ passing through $\alpha.$ The assertion follows.

We shall show that $M_i$ is locally contained in real analytic, totally real submanifolds lying in $\SSS.$ It is enough to show it for the point $\alpha.$

We have that $f:\Omega\rightarrow\CC^n$ is a diffeomorphism onto its image, real analytic, and $N=f(\Omega)$ is a totally real, real analytic, $(n-1)$-dimensional submanifold in $\CC^n$. We may assume that there exists a uniform constant $c>0$ such that $||f(x)||>c,$ for all $x\in \Omega.$ Set $f(x_0)=\alpha,$ for some $x_0.$ The function $F:\Omega\rightarrow \SSS$ defined by $F(x):=f(x)/||f(x)||$ is well defined, real analytic and $F(x_0)=f(x_0).$

A straightforward computation yields
$$F'(x)\cdot h=\frac{1}{||f(x)||}\Big(f'(x)\cdot h-\frac{f(x)}{||f(x)||^2}\mathrm{Re}\langle f'(x)\cdot h, f(x)\rangle\Big),$$
for all $x\in \Omega$ and $h\in \RR^{n-1}.$

We have the following dichotomy.

First, let us suppose that there exists $h_0\in \RR^{n-1}$ such that $\mathrm{Re}\langle f'(x_0)\cdot h_0, f(x_0)\rangle\neq0.$ Note that $f'(x_0)\cdot h_0\in T_{f(x_0)}N$ and recall that $T_{f(x_0)}\SSS=\{w\in \CC^n:\mathrm{Re}\langle w,f(x_0)\rangle=0\}.$ Hence, $f'(x_0)\cdot h_0\notin  T_{f(x_0)}\SSS.$ Since $\dim_\RR T_{f(x_0)}\SSS=2n-1$ ( it has codimension $1$),   we obtain that $ T_{f(x_0)}\SSS$ and $ T_{f(x_0)}N$ spann $ T_{f(x_0)}\CC^n,$ that is, $N$ and $\SSS$ intersect transversely, see Chapter 6, Section Transversality of \cite{Lee}.  Thus, $N_1:=N\cap \SSS\subset \SSS$ is a real analytic submanifold of $\CC^n$ of dimension $n-2.$  Moreover, $N\cap \SSS$ is totally real since it is contained in $N$ which is also totally real by assumtpion.

Secondly, let us suppose that $\mathrm{Re}\langle f'(x_0)\cdot h, f(x_0)\rangle=0,$ for all $h\in \RR^{n-1}.$ The condition $||f(x_0)||=1,$ yields  $F'(x_0)\cdot h=f'(x_0)\cdot h$ for all $h\in \RR^{n-1}.$ Since $f$ is diffeomorphism, $f'(x_0)$ has full rank $n-1,$ hence, $F'(x_0)$ has full rank $n-1$ too. Full rank is an open property, so it follows that there exists an open  neighbourhood $U$ of $x_0$ such that  $F'(x)$ has full rank $n-1$ for all $x\in U.$ Shrinking the open neighbourhood $U$ if necessary, we obtain that $F(U)$ is a real analytic submanifold in $\CC^n$ such that $F(U)\subset \SSS.$ Last, the equality 
$F'(x_0)\cdot h=f'(x_0)\cdot h$ for all $h\in \RR^{n-1},$ and the totally reality of $N$ yield  $T_{F(x_0)}F(U)\cap iT_{F(x_0)}F(U)=\{0\},$ hence, $F(U)$ is totally real at the point $F(x_0)$. Since totally reality is an open property there exists a smaller neighbourhood $U'$ of $x_0$ such that $N_2:=F(U')$ is a real analytic, totally real, $(n-1)$-dimensional submanifold of $\CC^n$ lying in $\SSS.$  

In both case $N_1,N_2,$ of the dichotomy, there exists $\epsilon>0$ such that $\overline{M_i}\cap \Ba(\alpha,\epsilon)$ is contained in a real analytic, totally real submanifold lying in $\SSS.$
\end{proof}

\begin{figure*}
\centering
\includegraphics[scale=0.72]{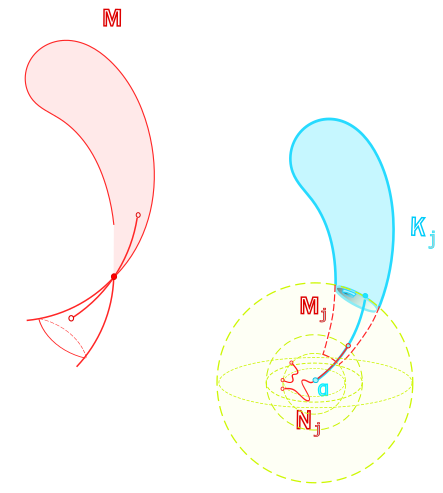}
\caption{}
\label{figure 2}
\end{figure*}

\subsection{Proof of Theorem~\ref{path conn peak sets}}
\begin{proof}
We follow Figure 2 which shows a two dimensional simple illustration. Let $M_i:=M$ and $\alpha\in \overline{M}\setminus M.$ Pick $\gamma:=\gamma_\alpha:(-\epsilon,\epsilon)\rightarrow\mathcal{Z}(p)\cap \SSS$ from Proposition~\ref{prop} which is contained in an one dimensional complex tangential $C^\infty$ submanifold in $\SSS.$

By assumption and the properties of Proposition~\ref{prop}, will exist $\epsilon>0$ and $\delta_0\in (0,1)$ such that for any open ball $\Ba(\alpha,\delta)\subset \CC^n,$ $\delta<\delta_0,$  we have that $\gamma([0,\epsilon))\cap \Ba(\alpha,\delta)$ is semi-algebraically connected,  $(\CC^n\setminus \Ba(\alpha,\delta_0))\cap \gamma([0,\epsilon))\neq \emptyset,$ and $(\CC^n\setminus \Ba(\alpha,\delta))\cap M$ is semi-algebraically connected. It follows that 
$$[\gamma([0,\epsilon))\cap \overline{\Ba(\alpha,\delta)}]\cup [(\CC^n\setminus \Ba(\alpha,\delta))\cap M]$$
is a semi-algebraically path connected set.

Pick sequences $\delta_0>\delta_1>\theta_1>\lambda_1>\delta_2>\theta_2>\lambda_2....>0$ converging to $0,$ and construct the following sequence of sets lying in $\overline{M}:$
$$K_j:=[\gamma([0,\epsilon))\cap \overline{\Ba(\alpha,\delta_j)}]\cup[(\CC^n\setminus \Ba(\alpha,\delta_j))\cap M].$$

The sequence $\{K_j\}$ increases. Indeed, if $\zeta\in K_j,$ then $\zeta\in\gamma([0,\epsilon))$ and $\zeta\in \overline{\Ba(\alpha,\delta_j)}$ or $\zeta\in\CC^n\setminus \Ba(\alpha,\delta_j)$ and $\zeta \in M$. The property $\zeta\in\CC^n\setminus \Ba(\alpha,\delta_j)$ and $\zeta \in M$ together with $\delta_j>\delta_{j+1}$ yield that $\zeta\in (\CC^n\setminus \Ba(\alpha,\delta_{j+1}))\cap M.$ On the other hand, let $\zeta\in\gamma([0,\epsilon))$ and $\zeta\in \overline{\Ba(\alpha,\delta_j)}.$ If $\zeta\in  \CC^n\setminus \Ba(\alpha,\delta_{j+1}),$ then $\zeta\in (\CC^n\setminus \Ba(\alpha,\delta_{j+1}))\cap \gamma([0,\epsilon))\subset (\CC^n\setminus \Ba(\alpha,\delta_{j+1}))\cap M,$ and hence, $\zeta\in K_{j+1}.$ If $\zeta\in \overline{\Ba(0,\delta_{j+1})},$ then $\zeta \in \gamma([0,\epsilon))\cap \overline{\Ba(0,\delta_{j+1})},$ from which follows again that $\zeta\in K_{j+1}.$ 

Each $K_j$ is a compact set. This follows from the fact that $(\CC^n\setminus \Ba(\alpha,\delta_j))\cap M=(\CC^n\setminus \Ba(\alpha,\delta_j))\cap \overline{M}$ is compact.

Furthermore, there exists a smooth map $\omega:(-\epsilon',\epsilon')\rightarrow \SSS$ such that $\omega'(t)\neq 0,$ for all $t\in (-\epsilon',\epsilon'),$ and $\gamma([0,\epsilon))\subset\omega((-\epsilon',\epsilon')).$ Pick $\epsilon>0$ such that $\omega((-\epsilon,\epsilon))\cap\Ba(\alpha,\delta)=[\omega((-\epsilon,0))\cup \gamma([0,\epsilon))]\cap \Ba(\alpha,\delta).$

We shall show that each $K_j$ is a peak set for $A^\infty(\BB)$ applying Lemma~\ref{FornLemma}. Indeed, pick $\epsilon_j>0$ such that $\omega((-\epsilon_j,0))\subset \Ba(\alpha,\lambda_j)$ and set $N_j:=[\omega((-\epsilon_j,0))\cup \gamma([0,\epsilon))]\cap \Ba(\alpha,\theta_j)$ and $M_j=M\cap \CC^n\setminus\overline{\Ba(\alpha,\lambda_j)}.$ All in all, we have that $K_j\subset \SSS$ is a compact set, $K_j\subset N_j\cup M_j$, the sets $N_j$ and $M_j$ are $C^\infty$ complex tangential manifolds in $\SSS,$  $\dim_\RR N_j<\dim_\RR M_j,$ the set $N_j\cap M$ is open in the relative topology of $N_j$ (since $N_j\cap M_j=\gamma(0,\epsilon)\cap [\Ba(\alpha,\theta_j)\setminus\overline{\Ba(\alpha,\lambda_j)}]$ which is open in $N_j$), the sets $N_j\cap K_j$ and $K_j\cap M_j$ are open in the relative topology of $K_j$ (since $\delta_j>\theta_j,$ we have $N_j\cap K_j=\Ba(\alpha,\theta_j)\cap \omega((-\epsilon,\epsilon))\cap\gamma([0,\epsilon))\cap \overline{\Ba(a,\delta_j)}\cup \emptyset=\Ba(\alpha,\theta_j)\cap\gamma([0,\epsilon))=\Ba(\alpha,\theta_j)\cap K_j$ which is relative open in $K_j,$ and $K_j\cap M_j=K_j\cap (\CC^n\setminus \overline{\Ba(\alpha,\lambda_j)})$ which is relative open in $K_j$). Lemma 4.1 yields that $K_j$ is a peak set for $A^\infty(\BB).$

\end{proof}

\subsection{Proof of Theorem~\ref{path conn peak sets2}}
\begin{proof}
Since $\alpha\in K$ and $K$ is a semi-algebraically path connected set which intersects $M_i$, there exists a semi-algebraic continuous map $\varphi:[0,1]\rightarrow K$ such that $\varphi(0)=\alpha$ and $\varphi((0,1])\subset M_i.$ Note that $\varphi([0,1])$ is an one dimensional semi-algebraic set, and hence, by the decomposition in Proposition 2.9.10 of \cite{Algebraic} and the arguments of Proposition~\ref{prop}, there exists $\gamma_\alpha:(-\epsilon,\epsilon)\rightarrow \mathcal{Z}(p)\cap \SSS$ satisfying the properties of Proposition~\ref{prop} with respect to an one dimensional component of $\varphi([0,1])$ that contains $\alpha$ in its relative boundary. Furthermore, $\gamma_\alpha([0,\epsilon])\subset K,$ namely, it is a compact subset of a peak set for $A^\infty(\BB).$ Shrinking $\epsilon>0$ if necessary,  Theorem 11 of \cite{Chaumat} and Theorem~\ref{peak contained locally} show that $\alpha$ is a regular point. 
\end{proof}

\subsection{Proof of Theorem~\ref{main2}}
\begin{proof}
By Theorem~\ref{non-cyclicity}, it is enough to identify cyclicity for $\alpha=n-\frac{1}{2}.$ By assumption, the arguments in the proof of Theorem~\ref{theoremm} and Lemma~\ref{LemmaPuri}, there exist finitely many compact sets $K_\ell\subset \SSS$ such that $\mathcal{Z}(p)\cap \SSS\subset \cup K_\ell,$ and $K_\ell$ is either an isolated point in the unit sphere with respect to $\mathcal{Z}(p)\cap \SSS,$ or there exist $g_\ell\in A^\infty(\BB)$ such that $\mathcal{Z}(g_\ell)\cap \overline{\BB}\equiv K_\ell$ and $g_\ell$ is cyclic for $D_{n-\frac{1}{2}}(\BB).$  Since $g_\ell$ are smooth in the closed unit ball, they are also multipliers in the space $\DDD_{n-\frac{1}{2}}(\BB).$ If $K_\ell$ is an isolated point then there exist a polynomial $g_\ell=\pi_\ell\circ U_\ell,$ where $\pi_\ell(z)=1-z_1,$ and $U_\ell$ is a unitary matrix such that $\mathcal{Z}(g_\ell)\cap\overline{\BB}\equiv K_\ell.$ The polynomial $g_\ell$ is cyclic for $\alpha\leq n,$ see Theorem~\ref{Perf} together with Theorem 11 of \cite{V}.

Set $g:=\prod_\ell g_\ell$ and note that $g$ is a zero-free in the unit ball smooth multiplier of $\DDD_{n-\frac{1}{2}}(\BB)$. Since product of cyclic multipliers is a cyclic function, we infer that $g$ is cyclic in $\DDD_{n-\frac{1}{2}}(\BB)$ too.

If $\zeta\in \mathcal{Z}(p)\cap \SSS,$ then $\zeta\in M_\ell,$ for some $\ell.$ Therefore, $\zeta\in K_\ell,$ and hence, $\zeta \in \mathcal{Z}(g)\cap \SSS.$ Whence, $\mathcal{Z}(p)\cap \SSS\subset \mathcal{Z}(g)\cap \SSS.$ By smoothness of $g$ we see that $g$ satisfies a Lipschitz condition of order $a > 0.$  Recall also that a polynomial is a multiplier for any Dirichlet-type space.

All the criteria of Theorem~\ref{Perf} are met, so we deduce cyclicity for $f$ in $\DDD_{n-\frac{1}{2}}(\BB).$ 
\end{proof}

\section{Further discussion}\label{sec:further}

Let us note that the real analytic map $\gamma_\alpha$ obtained in Proposition~\ref{prop} is a homeomorphism onto its image, ergo it cannot yield an one dimensional manifold with boundary.

It is possible to further classify the map $\gamma_\alpha$. As a matter of fact, straightforward computations on the complex tangential property and the real analytic expansion show that
$$a_{2\nu+k}=-\frac{\sum_{i=0}^{n-1}\Big(\sum_{l=0}^{k}(\nu+l)b_{\nu+l}^i\overline{b_{\nu+k-l}^i}\Big)}{2\nu+k}, \text{ } k=1,...,2\nu-1,$$
and
$$\sum_{i=0}^{n-1}\Big(\sum_{l=0}^{k}(\nu+l)b_{\nu+l}^i\overline{b_{\nu+k-l}^i}\Big)+(2\nu+k)a_{2\nu+k}+\sum_{l=0}^{k-2\nu}(2\nu+l)a_{2\nu+l}\overline{a_{2\nu+k-l}}=0, \text{  }k=2\nu,...$$

On the other hand, for the setting $n=2$, we have $\gamma_\alpha=(\gamma_1,1-\gamma_2)$. Hence, one may wonder whether is it possible to further classify the functions $h_j$ obtained through Weierstrass' Preparation Theorem and Puiseux's Parametrization Theorem in \cite{KV} in order to prove that $\alpha$ is not an essential singularity; this is an open
problem that merits further study in the near future. We also refer to \cite{Bickel} and \cite{KneseK} about the nature of these functions.

Let us also recall a multivariable version of Puiseux's Parametrization Theorem in $\CC^n$ by A. Parusi\'{n}ski, where the determinant of the corresponding monic polynomial of the stable polynomial is assumed to satisfy additional properties; see \cite{Parusinski}. In this specific case or under the graph assumptions of Theorem~\ref{cyclicity graph}, one may inquire about a characterization of cyclicity applying the radial dilation method appearing in \cite{KneseK,KV,V,Ziarati}.

We pose the two following questions that naturally stem from the contents of this work.

\begin{question*}[C]\phantomsection\label{Question C}
Under which additional hypotheses does a stable polynomial whose boundary zero set is a finite union of submanifolds of dimension at most one admits no  essential (resp. weak) singularities?
\end{question*}

\begin{question*}[D]\phantomsection\label{Question D}
Let $p\in \CC[z_1,...,z_n]$ be a stable polynomials and consider the partition $\mathcal{Z}(p)\cap \SSS=\cup_iM_i$ obtained from Theorem~\ref{zeroset}. To what extent does the cyclicity of $p$ in $\DDD_\alpha(\BB),$ $\alpha\in \RR,$ depend on the dimension of $\mathcal{Z}(p)\cap \SSS,$ or on the geometric nature of $M_i$?
\end{question*}

Notice that it is always possible to pick a stratification for $\mathcal{Z}(p)\cap\SSS,$ where the components may not be globally diffeomorphic to open hypercubes; see \cite[Proposition 9.1.8]{Algebraic}. So the definitions and the theory developed throughout the course of this article might merit greater generality. We focused mainly on the one dimensional boundary zeros. See also \cite{Sola1} and \cite{Knese} for information on singularities.

\section*{Acknowledgements} 
Discussions on the main ideas, in particular, on the conceptualisation of \hyperref[Question A]{Question (A)}, took place during the Bench Math Session workshop, February 2025, Krakow, Poland. The first named author would like to thank Filippo Fasoula, Voukoliki D. for enlightenments. The third named author is grateful to the camps of Avlona and Nea Santa for hosting him during the preparation of the manuscript. This work was supported by National Key R\&D Program of China, No. 2024YFA1015200 and the Natural Science Foundation of Guangdong Province, No. 2025A1515011428.

\bibliographystyle{plain}
	
\end{document}